\documentclass[12pt, reqno, a4paper,oneside]{amsart}
\usepackage{graphicx}
\usepackage{subfig}
\usepackage{color}
\usepackage{xcolor}
\usepackage[top=2.5cm,bottom=2.0cm,left=2.0cm,right=2.0cm]{geometry}
\usepackage{amsfonts, amsmath, amssymb, amsbsy, amsthm}
\usepackage{environments}
\usepackage{mathrsfs}
\usepackage{listings}
\usepackage{algpseudocode}
\usepackage{algorithm, algorithmicx}
\usepackage{diagbox}
\usepackage{bm}
\numberwithin{theorem}{section}
\numberwithin{equation}{section}

\definecolor{yscol}{HTML}{6622AA}
\definecolor{yzcol}{rgb}{0, 0.7, 0}
\definecolor{hwcol}{rgb}{0, 0, 0.9}
\definecolor{mlcol}{rgb}{0, 0.7, 0}
\definecolor{todocol}{rgb}{0.0, 0.4, 0.0}

\title[ \ ]{A General Framework of Linear Elasticity Enhanced Multiscale Coupling Methods for Crystalline Defects}

\author{Yanbo Zhan}
\address{Yanbo Zhan\\
	School of Mathematics\\
	Sichuan University\\
	No. 24 Yihuan Road\\
	Chengdu\\
	China
}
\email{2023322010008@scu.edu.cn}

\author{Yangshuai Wang}
\address{Yangshuai Wang\\
	Department of Mathematics\\
	Faculty of Science\\
	National University of Singapore\\
	10 Lower Kent Ridge Road\\
	Singapore
}
\email{yswang@nus.edu.sg}

\author{Hao Wang}
\address{Hao Wang\\
	School of Mathematics\\
	Sichuan University\\
	No. 24 Yihuan Road\\
	Chengdu\\
	China
}
\email{wangh@scu.edu.cn}

\date{\today}

\begin{document}
	
	\maketitle
	
	\def\rcut{r_{{\rm cut}}}
\def\bEa{\bar{\mathcal{E}}^{{\rm a}}}
\def\Ea{\mathcal{E}^{{\rm a}}}
\def\EaN{\mathcal{E}^{{\rm a},N}}

\def\Ya{\mathcal{Y}}
\def\Ua{\mathcal{U}}
\def\Rc{\mathcal{R}}
\def\ua{u^{{\rm a}}}
	\def\argm{{\rm argmin}}
	\def\Yn{\mathcal{Y}_{N}}
	\def\Un{\mathcal{U}_{N}}
	\def\yn{y^{N}}
	\def\un{u^{N}}
	\def\Ac{\Lambda_{\rm a}}
	\def\Cc{\Lambda_{\rm c}}
	\def\Ic{\Lambda_{\rm i}}
	\def\Ghc{\mathcal{G}}
	\def\Th{\mathcal{T}_{h}}
	\def\Nh{\mathcal{N}_{h}}
	\def\Uh{\mathcal{U}_{h}}
	\def\Yh{\mathcal{Y}_{h}}
	\def\YacNL{\mathcal{Y}_N}
	
	\def\yh{y^{{\rm qnll}}_{h}}
	\def\uh{u_{h}}
	\def\Ih{I_{h}^{{\rm L}}}
	\def\IhNL{I_{h}}
	\def\d{\text{d}}
	\def\Erfl{\mathcal{E}^{{\rm qnl}}}
	\def\Phia{\Phi^{{\rm a}}}
	\def\Phii{\Phi^{{\rm i}}}
	\def\PhiNLL{\Phi^{{\rm qnll}}}
	\def\Cnl{\Lambda_{{\rm nl}}}
	\def\Cl{\Lambda_{{\rm l}}}
	\def\El{\mathcal{E}^{{\rm qnll}}}
	\def\Wf{W_F}
	\def\Wpf{W^{\prime}_F}
	\def\Wppf{W^{\prime \prime}_F}
	\def\ynll{y^{{\rm qnll}}}
	\def\Pin{\Pi_{N}}
	\def\Pih{\Pi_{h}}
	\def\yrfl{y^{{\rm qnl}}}
	\def\DH{\textbf{\rm (DH)}}
	\def\CDH{C_{\bm{DH}}}
	\def\OmeA{\Omega_{{\rm a}}}
	\def\OmeI{\Omega_{{\rm i}}}
	\def\bOmeI{\overline{\Omega_{{\rm i}}}}
	\def\bOmeA{\overline{\Omega_{{\rm a}}}}
	\def\OmeC{\Omega_{{\rm c}}}
	\def\bOmeC{\overline{\Omega_{{\rm c}}}}
	\def\OmeNL{\Omega_{{\rm nl}}}
	\def\OmeL{\Omega_{{\rm l}}}
	\def\WL{W_{{\rm L}}}
	\def\c0{c_{0}}
	\def\Ki{\chi_{\xi,\rho}(x)}
	\def\conv{{\rm conv}}
	\def\Sa{S^{{\rm a}}(y;x)}
	\def\Srfl{S^{{\rm qnl}}(y;x)}
	\def\SNLL{S^{{\rm qnll}}(y;x)}
	\def\L{\bm{({\rm L})}}
	\def\S{\bm{({\rm S})}}
	\def\pGW{\partial_{F}W}
	\def\ppGW{\partial_{F}^{2}W}
	\def\pGWL{\partial_{F}W_{{\rm L}}}
	\def\ppGWL{\partial_{F}^{2}W_{{\rm L}}}
	\def\Rrfl{T^{{\rm qnl}}(y;x)}
	\def\RNLL{T^{{\rm qnll}}(y;x)}
	\def\yF{y^{F}}
	\def\vI{v_{\text{I}}}
	\def\vC{v_{\text{C}}}
	\def\va{v^{\text{a}}}
	\def\vc{v^{\text{c}}}
	\def\ganllF{\gamma^{{\rm qnll}}_{F}}
	\def\gaaF{\gamma^{{\rm a}}_{F}}
	\def\garflF{\gamma^{{\rm qnl}}_{F}}
	\def\Th{\mathcal{T}_{h}^{{\rm L}}}
	\def\ThNL{\mathcal{T}_{h}}
	\def\Nh{\mathcal{N}_{h}^{{\rm L}}}
	\def\NhNL{\mathcal{N}_{h}}
	\def\Nc{\mathcal{N}_{{\rm c}}}
	\def\Nnl{\mathcal{N}_{{\rm nl}}}
	\def\Nl{\mathcal{N}_{{\rm l}}}
	\def\Uh{\mathcal{U}_{h}^{{\rm L}}}
	\def\Yh{\mathcal{Y}_{h}^{{\rm L}}}
	\def\UhNL{\mathcal{U}_{h}}
	\def\YhNL{\mathcal{Y}_{h}}
	\def\Ph{\Pi_{h}^{{\rm L}}}
	\def\PhNL{\Pi_{h}}
	\def\tOmeL{\widetilde{\Omega}_{{\rm L}}}
	\def\tOmeC{\widetilde{\Omega}_{{\rm C}}}
	\def\tu{\tilde{u}}
	\def\vu{\textbf{u}}
	\def\vt{\textbf{t}}
	\def\vm{\textbf{m}}
	\def\vn{\textbf{n}}
	\def\vf{\textbf{f}}
	\def\vR{\textbf{R}}
	\def\R{\mathbb{R}}
	\def\Nb{\mathbb{N}}
	\def\Z{\mathbb{Z}}
	\def\C{\mathbb{C}}
	\def\P{\mathbb{P}}
	\def\V{\mathbb{V}}
	\def\E{\mathbb{E}}
	\def\U{\mathbb{U}}
	\def\kku{u^{(k-1)}}
	
	\def\ya{y^{{\rm a}} }
	\def\yai{\bar{y}^{{\rm a}}}
	\def\interface{{\rm interface}}

	\newtheorem{assumption}{Assumption}[section]
	\newtheorem{example}{Example}[section]
	\renewcommand{\theequation}{\arabic{section}.\arabic{equation}}
	\renewcommand{\thetheorem}{\arabic{section}.\arabic{theorem}}
	\renewcommand{\theassumption}{\arabic{section}.\arabic{assumption}}
	\renewcommand{\thedefinition}{\arabic{section}.\arabic{definition}}
	\renewcommand{\theexample}{\arabic{section}.\arabic{example}}

\begin{abstract}
	The atomistic-to-continuum (a/c) coupling methods, also known as the quasicontinuum (QC) methods,  are a important class of concurrent multisacle methods for modeling and simulating materials with defects.  The a/c methods aim to balance the accuracy and efficiency by coupling a molecular mechanics model (also termed as the atomistic model) in the vicinity of localized defects with the Cauchy-Born approximation of the atomistic model in the elastic far field. However, since both the molecular mechanics model and its Cauchy-Born approximation are usually a nonlinear, it potentially leads to a high computational cost for large-scale simulations. In this work, we propose an advancement of the classic quasinonlocal (QNL) a/c coupling method by incorporating a linearized Cauchy-Born model to reduce the computational cost. We present a rigorous {\it a priori} error analysis for this QNL method with linear elasticity enhancement (QNLL method), and show both analytically and numerically that it achieves the same convergence behavior as the classic (nonlinear) QNL method by proper determination of certain parameters relating to domain decomposition and finite element discretization. More importantly, our numerical experiments demonstrate that the QNLL method perform an substantial improvement of the computational efficiency in term of CPU times.
	
	
	
	%

\end{abstract}
	
	\section{Introduction}
\label{sec: introduction}


Crystalline defects, such as vacancies, dislocations, and cracks, play a crucial role in determining mechanical behavior of the materials~\cite{1996_AC_Ana_Solid_Defect_PMA,2009_Miller_Tadmor_Unified_Framework_Benchmark_MSMSE,2020_EG_Roadmap_Multiscale_Modeling_MSMSE,2012_Tadmor_Material_Con_Ato_Multiscale, chen2022qm, wang2021posteriori, chen2019adaptive}. While the defect core undergoes significant lattice distortions, the far-field region can be effectively modeled as an elastic force field~\cite{2020_JB_MD_CO_Thermo_Limit_Trans_Rate_Crystalline_Defect_ARMA,1984_Daw_Baskes_EAM_PRB, braun2022higher, olson2023elastic}. By leveraging this distinction, concurrent multiscale methods form an important class of methodology to simulating such systems. In order to achieve a balance between accuracy and computational efficiency, these methods decompose the computational domain into a small core region , which is modeled with detailed atomistic scale approaches such as molecular mechanics, and a large far-field region, which is often described by macroscopic continuum elasticity that can be further discretized by the finite element method to reduce the total degrees of freedoms (DoFs).

Concurrent multiscale methods may date back to 1980's and it was in \cite{1996_AC_Ana_Solid_Defect_PMA} that the Cauchy-Born approximation was introduced to the methods to initiate the atomisti-to-continuum (a/c) coupling methods, which are also well known as the quasicontinuum (QC) methods among the engineering community \cite{2003_RM_ET_QCM_JCAMD,2009_Miller_Tadmor_Unified_Framework_Benchmark_MSMSE,2013_ML_CO_AC_Coupling_ACTANUM, fu2023adaptive, fu2025meshac, wang2025posteriori}. Cauchy-Born approximation is essentially the continuum elasticity model derived the from a molecular mechanics model by certain homogenization process \cite{2002_PL_QCL_1D_MATHCOMP,2007_WE_PM_Cauchy_Born_Rule_and_The_Stability_ARMA,2007_PL_QCL_2D_SIAMNUM,2012_TH_CO_CB_Stability_Bravi_Lattice_M2NA, wang2021priori}. It is an important step in concurrent multiscale modeling to guarantee that the models of different scales produce the same result under elastic deformation, which is also known as the first level consistency of the method \cite{2012_WE_Principles_Multi_Model}. 

It was later discovered that the energy-based a/c coupling methods suffer from the so-called ``ghost force" at the interface between the regions where models of different scales transit \cite{1999_VS_RM_ETadmor_MOrtiz_AFEM_QC_JMPS}. Such phenomenon arises due to the nonlocality of the molecular mechanics model and the Cauchy-Born approximation which is continuum and thus local in nature. This is also known as the second level of consistency \cite{2012_WE_Principles_Multi_Model}. Considerable effort has been made in the recent two decades for the development and the mathematical analysis of the consistent a/c coupling methods so that the interface error is controllable and will not ruin the overall accuracy of the solution. Shimokawa et al. first proposed the Quasi-Nonlocal (QNL) method \cite{2004_Shimokawa_QCM_ErrAna_PRB}, which eliminates the "ghost force" by reconstructing the energy of atoms near the interface. E, Lu and Yang proposed the Geometric Reconstruction (GRC) method \cite{2006_WE_JL_ZY_GRC_PRB}, which is used to eliminate the "ghost force" in higher-dimensional cases. Ming and Yang, Dobson and Luskin conducted error analysis on the QNL method \cite{2009_PM_ZY_1D_QC_Nonlocal_MMS,2009_MD_ML_Optimal_Order_SIMNUM}. Ortner, Ortner and Wang performed a priori and a posteriori error analysis on the QNL and QNL with coarse-graining methods based on negative norms \cite{2011_CO_1D_QNL_MATHCOMP,2011_CO_HW_QC_A_Priori_1D_M3AS}. Ortner and Zhang proposed the two-dimensional geometric reconstruction-based consistent atomistic/continuum (GRAC) method~\cite{2012_CO_LZ_GRAC_Construction_SIAMNUM,2014_CO_LZ_GRAC_Coeff_Optim_CMAME}, which extends the QNL method to two-dimensional cases. We refer to \cite{2013_ML_CO_AC_Coupling_ACTANUM} for a thorough review of the QNL method, and we will follow its theoretical framework to complete some of the proofs in this paper.

The a/c methods improves the computational efficiency compared with the full molecular mechanics simulations in terms of the substantial reduced DoFs as a result of the finite element discretization of the continuum elasticity model. Therefore existing works concentrate mainly on the convergence rate for different a/c methods, which illustrate that for a given error how many DoFs are required or vice versa \cite{1980_MB_CM_WW_Dislocation_Near_Crack_JM,1982_MM_MD_Crack_Boundary_PMA,1989_HF_HE_MP_Metals_Composite_Materials_ZM,1991_SK_PG_HF_Crack_FEM_PMA} (cf. Figure \ref{fig: convergence_QNL_QNLL_alpha12_CG}, \ref{fig: convergence_QNL_QNLL_alpha10_CG} and \ref{fig: convergence_QNL_QNLL_alpha08_CG} in the current paper). However, the molecular mechanics model we employ in the core region is usually highly nonlinear and thus is the corresponding Cauchy-Born approximation in the elastic far field as well as the a/c coupling model in whole. The high nonlinearity of the a/c model may still results in an excessive computational cost, especially in large scale simulations, which motivates this study. 

%
%

The current work focuses on developing and analyzing an energy-based a/c coupling methods that integrates the classic (nonlinear) quasinonlocal (QNL) method \cite{2004_Shimokawa_QCM_ErrAna_PRB,2011_CO_1D_QNL_MATHCOMP,2012_XL_ML_Finite_Range_QNL_IMANUM} with the linearized Cauchy-Born approximation, which we name as the QNLL method. The purpose of such development is to enhance computational efficiency of the a/c method through the proper application of linearized model in the elastic far-field region while retain the same level of accuracy as the fully nonlinear coupling scheme. We then perform a rigorous {\it a prior} analysis of the QNLL method. We identify that the optimal rate of convergence of the method can be achieved by choosing certain parameters according to the regularity of the defect equilibrium. In particular, by balancing the lengths of the regions where different models apply and the size of the finite element discretization, we obtain a sharp error estimate of the QNLL method which retain the same order of convergence as that of the fully nonlinear QNL method. We validate our theoretical results through a series of numerical experiments. In addition, our numerical experiments demonstrate that QNLL achieve a substantial improvement of computational efficiency compared with the fully nonlinear QNL method in terms of CPU times.

To set out our ideas in a clear way, we consider a one dimensional atomistic system with nearest and next-nearest neighbor multibody interactions. However, we note that the error analysis and balancing techniques proposed are general and are expected to be extended to higher dimensions problems and other fully nonlinear a/c coupling methods with finite interaction ranges~\cite{wang2024theoretical, ortner2023framework}. This study serves as a proof of concept, laying the foundation for possible further extension, which are discussed in the conclusion.

\subsection{Outline}
\label{sec: outline}

The paper is organized as follows. Section~\ref{sec: qnll_model} introduces the atomistic model, the quasinonlocal (QNL) method, and the QNL method with the linearized Cauchy-Born (QNLL) coupling scheme proposed in this work. In Section~\ref{sec: anal_qnll_ncg}, we present an {\it a priori} error analysis for the QNLL method without coarse graining and propose a balancing method to adjust the lengths of various regions based on the error estimate. Numerical experiments are provided to validate the balancing method and highlight the computational efficiency of the QNLL method. Section~\ref{sec: qnll_cg} extends the analysis to the QNLL method with coarse graining in the continuum region. We provide an error estimate and propose a balancing method that includes the adjustment of the finite element mesh length. Numerical experiments demonstrate the effectiveness of this method and show that, despite the reduced degrees of freedom due to coarse graining, the QNLL method can still retain its computational efficiency. Section~\ref{sec: conclusion}  concludes the paper and outlines potential future directions. For clarity and brevity, some detailed proofs are provided in the appendices.

\subsection{Notations}
\label{sec: notations}
We use the symbol $\langle\cdot,\cdot\rangle$ to denote an abstract duality
pairing between a Banach space and its dual. The symbol $|\cdot|$ normally
denotes the Euclidean or Frobenius norm, while $\|\cdot\|$ denotes an operator
norm. We denote $A\backslash\{a\}$ by
$A\backslash a$, and $\{b-a~\vert ~b\in A\}$ by $A-a$. Directional derivatives are denoted by $\nabla \rho f := \nabla f\cdot\rho,\, \rho \in \mathbb{R}$. For $E \in C^2(X)$, the first and second variations are denoted by
$\langle\delta E(u), v\rangle$ and $\langle\delta^2 E(u) v, w\rangle$ for $u,v,w\in X$.

We write $|A| \lesssim B$ if there exists a constant $C$ such that $|A|\leq CB$, where $C$ may change from one line of an estimate to the next. When estimating rates of decay or convergence, $C$ will always remain independent of the system size, the configuration of the lattice and the the test functions. The dependence of $C$ will be normally clear from the context or stated explicitly. 


	
	\section{Atomistic to Nonlinear-Linear Elasticity Coupling Method}
\label{sec: qnll_model}

In this section, we present the atomistic-to-continuum (a/c) coupling methods in one dimension. While the extension to higher dimensions is relatively straightforward, it involves significantly more complex notations. To ensure clarity and brevity, we focus on the one-dimensional case. Section~\ref{sec: introduction_atom} introduces the atomistic model, which serves as the reference framework. Section~\ref{sec: introduction_qnl} provides an overview of the classical nonlinear quasinonlocal (QNL) method. Finally, Section~\ref{sec: introduction_qnll} details the linearization of the Cauchy-Born model and the QNL method with the Linearized Cauchy-Born (QNLL) formulation, which is the main focus of this work.

\subsection{Atomistic model}
\label{sec: introduction_atom}

We consider an infinite atomistic chain or one dimensional crystal lattice indexed by $\Z$. The reference configuration is given by $F\Z$, where $F>0$ is a macroscopic strain. 
We define the space of the displacements and the {\it admissible} set of deformations by
\begin{align*}
	\Ua&:=\{u:\Z\rightarrow\R ~|~ \nabla u \in L^{2} \},\\
	\Ya&:=\{y(x)= Fx+u(x) ~|~ u \in \Ua\}.
\end{align*}

Let $\rcut>0$, we fix an interaction range $\Rc := \{\pm1,\dots,\pm\rcut\}$. For each $y \in \Ya$ and $\xi \in \Z$, we define the finite difference stencil 
\begin{equation*}
	Dy(\xi):= \big(D_{\rho}y(\xi)\big)_{\rho \in \mathcal{R}}, \quad \text{where} ~~ D_{\rho}y(\xi):= y(\xi + \rho )-y(\xi).
\end{equation*}
For simplicity, we fix $\rcut = 2$ throughout this work. However, the analysis can be readily extended to a general interaction range.

Let $V\in C^{3}(\R^{\Rc})$ be the interatomic many-body site potential. For a deformation $y\in\Ya$, we define the energy of the infinite atomistic model by
\begin{equation}\label{All-Atomistic Energy}
	\bEa (y):= \sum_{\xi\in\Z}\Phia_{\xi}(y) := \sum_{\xi\in\Z} \left[V\big(Dy(\xi)\big)-V(F\Rc)\right],
\end{equation}
where $\Phia_{\xi}(y)$ is the site energy (per atom energy contribution~\cite{2013_ML_CO_AC_Coupling_ACTANUM}) for the site $\xi$.

For analytical purpose, we assume the regularity of site potential $V$, i.e.,
\begin{align*}
		m(\bm{\rho})&:= \prod_{i =1}^{j} \vert \rho_{i} \vert \sup_{\bm{g} \in \R^{\mathcal{R}}} \Vert V_{\rho}(\bm{g}) \Vert \quad \text{for} \ \bm{\rho} \in \mathcal{R}^{j}, \ \text{and} \\
		M^{(j,s)}&:= \sum_{\bm{\rho} \in \R^{j} } m(\bm{\rho}) \vert \bm{\rho} \vert^{s}_{\infty},
\end{align*}
where $V_\rho$ is the derivative of the energy functional $V$ with respect to $\rho$ and $\vert \bm{\rho} \vert_{\infty}:=\max_{i= 1,\dots,j}\vert \rho_{i} \vert$.

By the definition of $\Ya$ and $\Ua$, $\Ea$ is well-defined on $\Ya$, which means that the sum on the infinite lattice is actually finite \cite[Proposition 3.7]{2013_ML_CO_AC_Coupling_ACTANUM}. Given a dead load $f \in \Ya$, the atomistic model is defined by the following minimization problem 
\begin{equation}\label{All-Atomistic solution}
	\yai \in \arg \min\{\bEa(y) -\langle f,y\rangle_{\Z}\ | \ y \in \Ya\},
\end{equation}
where $\langle f,y\rangle_{\Z} = \sum_{\xi \in \Z} f(\xi)y(\xi)$ and ``$\arg\min$'' is understood as the set of local minimizers. Furthermore, we denote the local minimizer of the displacement as $\bar{u}^{\rm a}:=\bar{y}^{\rm a}-Fx$. 

If $\yai$ solves \eqref{All-Atomistic solution}, then it satisfies the first order optimality condition 
\begin{equation}\label{All-Atomistic solution condition}
	\langle \delta \bEa(\yai),v \rangle = \langle f,v\rangle_{\Z}, \quad \forall v \in \Ua.
\end{equation}
In addition to \eqref{All-Atomistic solution condition}, if $\yai \in \Ya$ satisfies the second-order optimality condition which is given by
\begin{equation}\label{All-Atomistic strong local minimizer}
	\langle \delta^{2}\bEa(\yai)v,v\rangle \ge c_{0} \Vert \nabla v \Vert^{2}_{L^{2}},  \quad \forall v\in \Ua,
\end{equation}
for some $c_{0} >0$, we say that the solution $\yai$ is a strongly stable local minimizer.

	
	To model practical defects and establish a foundation for consistency and stability analysis~\cite{2016_EV_CO_AS_Boundary_Conditions_for_Crystal_Lattice_ARMA}, we introduce the decay hypothesis, which describes the decay and regularity of local minimizers (equilibrium). This assumption ensures smooth asymptotic behavior, which is crucial for the subsequent analysis.
	
	
	\textbf{Decay Hypothesis $\DH$}: There exists a strong local minimizer $\bar{u}^{\rm a} \in \Ua$ and $\alpha > 1/2$, for $x$ sufficiently large such that
	\begin{equation}\label{Decay Hypothesis}
		\vert \nabla^{j} \bar{u}^{\rm a}(x) \vert \le \CDH x^{-\alpha+1-j}, \quad j = 0,1,2,3,
	\end{equation}
	where $\CDH>0$ is a constant that depends on lattice and interatomic potentials.

	\subsection{Classic quasinonlocal (QNL) methods}
	\label{sec: introduction_qnl}
	
	The atomistic problem defined by \eqref{All-Atomistic solution} is computationally intractable due to its formulation on an infinite lattice, its reliance on a nonlocal interaction potential across the entire domain, and the fact that each atom is treated as a separate degree of freedom. To address these challenges, atomistic/continuum (a/c) coupling methods employ domain decomposition strategies to truncate the infinite computational domain, introduce a reduced model in certain regions, and further decrease the number of degrees of freedom through coarse graining. In this section, we illustrate the construction of a/c methods using the classical quasinonlocal (QNL) method~\cite{2011_CO_1D_QNL_MATHCOMP}, which serves as the foundation for the developments presented in subsequent sections.
	
	The a/c coupling methods often make the following three steps of approximations. The first step is to truncate the infinite lattice to a finite domain on which the computation is carried out. The second step is to derive a {\it local} and continuum approximation for the {\it nonlocal} and discrete atomistic model. The third step is to decompose the domain so that the atomistic and the continuum models are properly utilized in different regions and to make a special treatment at the interfaces where the two different models meet so that nonphysical phenomenon is avoided. We now elaborate the construction of the QNL method according to the aforementioned three steps.
	
	\subsubsection{Truncation}
	
	We first truncate the infinite domain simply by fixing an $N\in \Nb$ and define the truncated computational domain to be $\Omega:=[-N,N]$. The set of lattice inside the computational domain is given by $\Lambda := \Omega \cap \Z = \{-N, \ldots, -1, 0, 1, \ldots, N\}$.


	
	
	Though the first step is very easy, we pause here to introduce an auxiliary problem that simplifies our error analysis. Specifically, we apply a Dirichlet boundary condition on $\Omega$ and define the finite dimensional space of displacements and the corresponding admissible set of deformation as
	\begin{align*}
		\Un&:= \{u\in\Ua \ | \ u(\xi) = 0 \text{ for }\xi \le -N \text{ or } \xi \ge N\},\\
		\Yn&:= \{y(x) = Fx + u(x) \ | \ u \in \Un\}.
	\end{align*}
	We can then define the truncated atomistic method (often denoted by ATM \cite{2016_EV_CO_AS_Boundary_Conditions_for_Crystal_Lattice_ARMA,2013_ML_CO_AC_Coupling_ACTANUM}) as 
	\begin{equation}\label{Atomistic solution condition}
		\ya \in \arg \min \{\Ea(y) - \langle f,y \rangle_{N} \ | \ y \in \Yn \},
	\end{equation}
	where $\Ea(y): =  \sum_{\xi =-N}^{N}\Phia_{\xi}(y)$ and $\langle f,y\rangle_{N} = \sum_{\xi =-N}^{N} f(\xi)y(\xi)$, which can be considered as a Galerkin approximation of the original atomistic model \eqref{All-Atomistic solution}. 
	
	Since our atomistic/continuum coupling methods are all defined on the finite domain $\Omega$, we can bound the error between the solution of any of the a/c methods $y^{\rm ac}$ and that of original atomistic model $\yai$ by 
	\begin{equation}
		\label{eq: error separation}
		\| y^{\rm ac} - \yai \| \le \|y^{\rm ac} - \ya \| + \|\ya - \yai \|.
	\end{equation} 
	We will only concentrate on the first part of the right hand side of \eqref{eq: error separation} in the error analysis and the estimate of the second part, which is essentially the truncation error, is given by the following lemma \cite[Theorem 3.14]{2013_ML_CO_AC_Coupling_ACTANUM}:
	
	\begin{lemma}
		Let $y^{{\rm a}}$ be a strong local minimizer of \eqref{All-Atomistic solution condition} satisfying $\DH$ and \eqref{All-Atomistic strong local minimizer}. Then there exists $N_{0} \in \Nb$ such that, for all $N \ge N_{0}$, there exists a strong local minimizer $\ya$ of \eqref{Atomistic solution condition} satisfying
		\begin{equation}\label{Truncation error}
			\Vert \nabla \yai - \nabla \ya\Vert_{L^{2}} \le \frac{16M^{(2,0)}\CDH}{\sqrt{2\alpha-1}} N^{\frac{1}{2}-\alpha}.
		\end{equation}
	\end{lemma}
	


	\subsubsection{Continuous approximation}
	
	Next, we consider the second approximation, which involves using a local continuum model to approximate the original nonlocal atomistic model. To further reduce the number of degrees of freedom, we apply the continuum approximation to the atomistic potential \eqref{All-Atomistic Energy} based on the Cauchy-Born rule~\cite{2013_CO_FT_Cauchy_Born_ARMA}, where we define the strain energy density from the interaction potential $V$ as 
	\begin{equation}
		\label{Cauchy-Born site energy}
		W(F) := V(F\Rc).
	\end{equation}

	\subsubsection{Domain decomposition}
	
	Finally, we will achieve the third approximation through the atomistic-to-continuum coupling method. The Atomistic-to-Continuum coupling method achieves an quasi-optimal trade-off between computational cost and accuracy by coupling the atomistic model with its continuum approximation. In this paper, the coupling scheme we use is the reflection method, which is show to be universally stable  \cite{2014_CO_AS_LZ_Stabilization_MMS}, and we will discuss in stability analysis. Because we fix $\rcut = 2$, the reflection method is equivalent to the Quasi-nonlocal method (QNL method) in this case \cite{2013_ML_CO_AC_Coupling_ACTANUM}. The QNL method eliminates the ``ghost force" term by providing a precise definition of the energy on the interface \cite{2011_CO_1D_QNL_MATHCOMP}. Therefore, for the sake of simplicity in naming, the remaining sections of this paper will use the QNL method to uniformly refer to the reflection method.
	
	We first of all decompose the computational domain $\Omega$ into three different regions. The first one is the atomistic region $\OmeA$ in which we assume the defect core is contained. The second one is the continuum region $\OmeC$ where we assume the deformation is smooth enough. The third one is th interface region $\OmeI$ which consists of a small number of layers of atoms between $\OmeA$ and $\OmeC$. The lattice points in the three regions are defined by $\Ac:= \OmeA \cap \Lambda$, $\Cc:= \OmeC \cap \Lambda$ and $\Ic := \OmeI \cap \Lambda$. To be more specific, we denote them by (for simplicity, we denote $\bar{K} = K+\rcut$)
	\begin{align*}
		\Ac &:=\{-K,-K+1,\dots,K\}, \qquad \qquad \qquad \qquad \qquad \qquad \qquad \quad ~~  \text{(Atomistic region)},\\
		\Ic &:=\{-\bar{K},-\bar{K}+1\}\cup \{\bar{K}-1,\bar{K}\}, \qquad \qquad \qquad \qquad \qquad \qquad ~~ \text{(Interface region)},\\
		\Cc &:=\{-N,-N+1,\dots,-\bar{K}-1\}\cup\{\bar{K}+1,\bar{K}+2,\dots,N\}, \qquad ~ \text{(Continuum region)}.
	\end{align*}
	
	Since we will calculate the integral on continuum region, we define the following three continuous intervals:
	\begin{align*}
		\OmeA &=[-K-1,K+1],\\
		\OmeI &=[-\bar{K},-\bar{K}+1] \cup [\bar{K}-1,\bar{K}],\\
		\OmeC&=[-N,-\bar{K}]\cup [\bar{K},N].
	\end{align*}
	
	%
	%
	%

	Moreover, we define, for any $x\in (\xi-1,\xi), \nabla y(x) = y(\xi)-y(\xi-1)$, which is a piecewise constant interpolation of $y$ with respect to lattice sites. The specific construction of the QNL model is given as follows:
	\begin{equation}\label{QNL energy}
		\Erfl (y)=\sum_{\xi \in \Ac}\Phia_{\xi}(y) + \sum_{\xi \in \Ic}\Phii_{\xi}(y) + \int_{\OmeC} W(\nabla y)\,\d x,
	\end{equation}
	where the interface energy reads
	\begin{align*}
		\Phii_{-L}(y)&=V(-D_{2}y,-D_{1}y,D_{1}y,D_{2}y)+\int_{-L-1}^{-L-\frac{1}{2}}W(\nabla y)\,\d x,\\
		\Phii_{-L+1}(y)&=V(2D_{-1}y,D_{-1}y,D_{1}y,D_{2}y),\\
		\Phii_{L-1}(y)&=V(D_{-2}y,D_{-1}y,D_{1}y,2D_{1}y),\\
		\Phii_{L}(y)&=V(D_{-2}y,D_{-1}y,-D_{-1}y,-D_{-2}y)+\int_{L+\frac{1}{2}}^{L+1}W(\nabla y)\,\d x.
	\end{align*}
	
	Given a dead load $f\in \Ya$, we seek solution
	\begin{equation}\label{QNL solution conditon}
		\yrfl \in \arg \min \{\Erfl (y) - \langle f,y\rangle_{N} \ | \ y \in \Yn\}.
	\end{equation}
	Similarly, if $\yrfl$ solves \eqref{QNL solution conditon}, then it satisfies the first order optimal condition
	\begin{equation}\label{QNL solution first order optimality condition}
		\langle \delta \Erfl(\yrfl ),v \rangle = \langle f,v\rangle_{N}, \quad \text{for all }v \in \Un.
	\end{equation}
	
	\subsection{QNL with linearized Cauchy-Born (QNLL) method}
	\label{sec: introduction_qnll}
	
	In this section, we will introduce the QNL Methods with Linearized Cauchy-Born (QNLL). The idea behind QNLL method is to replace a nonlinear elasticity model with a linear elasticity model to further simplify the computation. To that end, we make further approximations: based on the QNL method, we use a linear elasticity model to approximate the nonlinear elasticity model. Compared to the nonlinear model, the linear model can effectively reduce computational costs. However, this also introduces new errors, which will be analyzed in the next section.
	
	First, we construct the linear elasticity model. Here, we obtain the linear elasticity model by performing a Taylor expansion of the Cauchy-Born energy \eqref{Cauchy-Born site energy} around the uniform deformation. After introducing linearization, we will introduce the nonlinear-linear elasticity coupling (QNLL) model, which essentially replaces a portion of the nonlinear continuum energy functional $W$ in a continuum region $\OmeC$ with a linear energy functional $\WL$. Here, we first provide the form of $\WL$. For simplicity, we denote $W^{(k)}_{F} = W^{(k)}(F),\ k =0,1,2$. We use Taylor's expansion in order to linearize Cauchy-Born strain energy density: 
	\begin{equation*}
		W(\nabla y) = \Wf+\Wpf\nabla u+\frac{1}{2}\Wppf(\nabla u)^{2} +\dots.
	\end{equation*}
	We then define the linearized Cauchy-Born strain energy density:
	\begin{equation*}
		\WL(\nabla y) = \Wf+\Wpf\nabla u+\frac{1}{2}\Wppf(\nabla u)^{2}.
	\end{equation*}
	
	Next, based on the region partitioning of the QNL method, we will further subdivide $\Cc$ into nonlinear and linear regions. We denote them by
	\begin{align*}
		\Cnl &= \{- L,L+1,\dots,-\bar{K}-1\}\cup\{\bar{K}+1, \bar{K}+2,\dots,L\},\\
		\Cl &=\{-N,-N+1,\dots,L-1\}\cup \{L+1,L+2,\dots,N\}.
	\end{align*}
	Similarly, we will also divide $\OmeC$ into nonlinear and linear parts:
	\begin{align*}
		\OmeNL &=[-L,-\bar{K}] \cup [\bar{K},L],\\
		\OmeL &=[-N,L] \cup [L,N].
	\end{align*}
	
	In the linearized continuouum region, by approximating the nonlinear elasticity model with the linear elasticity model, we obtain the QNLL model:
	\begin{align}\label{Nonlinear-linear energy}
		\El (y) &= \sum_{\xi = -N}^{N}\PhiNLL_{\xi}(y) \nonumber \\
		&=\sum_{\xi \in \Ac}\Phia_{\xi}(y) + \sum_{\xi \in \Ic}\Phii_{\xi}(y) + \int_{\OmeNL} W(\nabla y)\,\d x + \int_{\OmeL} \WL(\nabla y)\,\d x.
	\end{align}

	Given a dead load $f\in \Ya$, we seek
	\begin{equation}\label{Nonlinear-linear solution condition}
		\ynll \in \arg \min \{\El (y) - \langle f,y\rangle_{N} \ | \ y \in \Yn\}.
	\end{equation}
	If $\ynll$ solves \eqref{Nonlinear-linear solution condition}, then it satisfies the first order optimal condition
	\begin{equation}\label{QNLL solution first order optimality condition}
		\langle \delta \El(y),v \rangle = \langle f,v\rangle_{N}, \quad \text{for all }v \in \Un.
	\end{equation}
	
	\section{A Priori Analysis for the QNLL Method}
\label{sec: anal_qnll_ncg}


In this section, we will provide the {\it a priori} error estimate for the atomistic to nonlinear-linear elasticity coupling model $\Vert \nabla \yai - \nabla \ynll \Vert_{L^{2}}$, following analytical framework shown in~\cite{2013_ML_CO_AC_Coupling_ACTANUM,2011_CO_1D_QNL_MATHCOMP,2011_CO_HW_QC_A_Priori_1D_M3AS}. We will first present the consistency error estimate and stability analysis results of the QNLL model, and then give the {\it a priori} error estimate based on the inverse function theorem (Lemma \ref{Inverse function theorem}) provided in the Appendix. The detailed consistency error analysis is provided in Section \ref{sec: consistency_qnll_ncg}, stability analysis in Section \ref{sec: stability_qnll_ncg}, and the final {\it a priori} error estimate in Section \ref{sec: priori_qnll_ncg}. Finally, in Section~\ref{sec: balance_of_qnll_ncg_model}, based on the {\it a priori} error estimate and the $ \DH $ assumption, we propose a method to balance the lengths of the regions in the QNLL model to achieve the same convergence order as the QNL model.

\subsection{Consistency error}
\label{sec: consistency_qnll_ncg}

In the consistency error estimate, we decompose the total error using the triangle inequality into two components: (i) the coupling error, which quantifies the difference between the atomistic model and the QNL model, and (ii) the linearization error, which arises from the transition from the QNL model to the QNLL model. By estimating these two errors separately, we derive the overall consistency error between the atomistic model and the QNLL model. Specifically, we have
\begin{align*}
	T(\ya) &= \langle \delta \Ea (\ya),v\rangle -\langle \delta \El (\ya),v\rangle\\
	&=\langle \delta \Ea (\ya),v\rangle -\langle \delta \Erfl (\ya),v\rangle ~~\qquad \text{(the coupling error)}\\
	&\ \ +\langle \delta \Erfl (\ya),v\rangle -\langle \delta \El (\ya),v\rangle. \quad \text{(the linearization error)}
\end{align*}
In the following, we establish an error bound for $T(y^{\rm a})$.

\subsubsection{QNL coupling error}

We define the weighted characteristic function of a bond $(\xi ,\xi +\rho)$ by
\begin{equation}\label{Definition of Ki}
	\Ki:=\left\{
	\begin{aligned}
		&\vert\rho\vert^{-1}, &x\in \text{int} (\conv \{\xi,\xi+\rho\}), \\
		&\frac{1}{2}\vert\rho\vert^{-1}, &x\in \{\xi,\xi+\rho\}, \\
		&0, &\text{otherwise}.
	\end{aligned}
	\right.
\end{equation}
We then obtain for $D_{\rho}v(\xi) = v(\xi+\rho) -v(\xi)$ that
\begin{equation}\label{Diff to int}
	D_{\rho}v(\xi)=\int_{\xi}^{\xi+\rho}\frac{\rho}{\vert\rho\vert}\nabla v\,\d x
	=\int_{\R}\rho \Ki\nabla v\,\d x.
\end{equation}

The first variation of the atomistic energy functional \eqref{All-Atomistic Energy} at $y \in \Yn$ is given by
\begin{equation*}
	\langle \delta \Ea(y), v\rangle = \sum_{\xi = -N}^{N} \sum_{\rho \in \Rc}\Phia_{\xi,\rho}(y)D_{\rho}v(\xi).
\end{equation*}
We apply \eqref{Diff to int}, it follows that
\begin{equation}\label{Atomistic first variation}
	\langle \delta \Ea(y), v\rangle = \int_{-N}^{N} \Sa \nabla v\,\d x.
\end{equation}
where
\begin{equation}\label{Atomistic stress tensor}
	\Sa = \sum_{\xi=-N}^{N} \sum_{\rho \in \Rc} \rho \Ki \Phia_{\xi,\rho}(y).
\end{equation}



The first variation of the energy functional of QNL approximation defined by \eqref{QNL energy}, for any $v\in \YacNL$, we have 
\begin{equation}\label{QNL first variation}
	\begin{split}
		\langle \delta \Erfl (y), v\rangle &= \sum_{\xi \in \Ac} \sum_{\rho \in \Rc} \Phia_{\xi,\rho}(y)D_{\rho}v(\xi) + \int_{\OmeC}\partial_{F} W(\nabla y)\nabla v\,\d x \\
		&=: \int_{-N}^{N} \Srfl \nabla v \,\d x.
	\end{split}
\end{equation}
where
\begin{equation}\label{QNL stress tensor}
	\Srfl := \left\{
	\begin{aligned}
		&\sum_{\xi \in \Ac}\sum_{\rho \in \Rc}\rho \Ki \Phia_{\xi,\rho}(y) +\sum_{\xi \in \Ic}\sum_{\rho \in \Rc}\rho \Ki \Phii_{\xi,\rho}(y), &x\in \OmeA \cup \OmeI, \\
		&\partial_{F}W(\nabla y), &x\in \OmeC.
	\end{aligned}
	\right.
\end{equation}


We now define the error in QNL model stress as
\begin{equation*}
	\Rrfl := \Srfl -\Sa.
\end{equation*}
Next, we will provide a point estimate of the stress tensor for the QNL model. For the sake of presentation simplicity, we leave the detailed proof, please refer to Appendix \ref{Appendix section 1}.

\begin{proposition}\label{Pointwise coupling stress tensor}
	Let $y\in\Ya, x \in [-N,N]$, then 
	\begin{equation*}
		\vert \Rrfl \vert \lesssim \left\{
		\begin{aligned}
			&0, &x \in \OmeA \backslash \bOmeA,\\
			&M^{(2,1)}\Vert \nabla^{2}u\Vert_{L^{\infty}(v_{x})}, &x\in \OmeI \cup \bOmeA,\\
			&M^{(2,2)}\Vert \nabla^{3}u\Vert_{L^{\infty}(v_{x})}+M^{(3,2)}\Vert \nabla^{2}u\Vert_{L^{\infty}(v_{x})}^{2}, &x \in \OmeC
		\end{aligned}
		\right.
	\end{equation*}
	where $v_{x}:=\big[\lfloor x\rfloor+1-2\rcut,\lfloor x\rfloor+2\rcut\big]$ is the neighborhood of some $x\in \R$, the meaning of $\lfloor x\rfloor$ here is to take the floor of $x$, which is the greatest integer less than or equal to $x$, and $\bOmeA = [-K,-K+\rcut]\cup [K-\rcut,K]$.
\end{proposition}

Based on the point estimate of $\Rrfl$ provided above, we will give the consistency error estimate for the QNL model. Please refer to Appendix \ref{Appendix section 2} for the proof.

\begin{proposition}\label{Coupling consistency error estimate}
	We split coupling error into two parts:
	\begin{equation*}
		\int_{-N}^{N}\Rrfl\nabla v\,\d x=\int_{\OmeI\cup\bOmeA}\Rrfl\nabla v\,\d x-\int_{\OmeC}\Rrfl\nabla v\,\d x.
	\end{equation*}
	
	For any $v\in \Yn$, we have
	\begin{equation}\label{Interface region stress tensor}
		\int_{\OmeI\cup\bOmeA}\Rrfl\nabla v\,\d x \lesssim M^{(2,1)}\Vert\nabla^{2}u\Vert_{L^{2}(\bOmeI)}\Vert\nabla v\Vert_{L^{2}(\OmeI\cup\bOmeA)},
	\end{equation}
	and
	\begin{equation}\label{Continuum region stress tensor}
		\int_{\OmeC}\Rrfl\nabla v\,\d x \lesssim(M^{(2,2)}\Vert \nabla^{3}u\Vert_{L^{2}(\bOmeC)}+M^{(3,2)}\Vert \nabla^{2}u\Vert_{L^{4}(\bOmeC)}^{2})\Vert\nabla v\Vert_{L^{2}(\OmeC)},
	\end{equation}
	where $\bOmeI = [-\bar{K}+1-2\rcut,-\bar{K}+4+2\rcut]\cup[\bar{K}-3-2\rcut,\bar{K}+2\rcut], \bOmeC=[-N-2\rcut,-\bar{K}-1+2\rcut]\cup[\bar{K}+1-2\rcut,N+2\rcut]$.
\end{proposition}

	

	
	

\subsubsection{Linearization error}

Firstly, we calculate the first variation of the energy functional of Nonlinear-linear elasticity coupling energy defined by \eqref{Nonlinear-linear energy}, for any $v\in\YacNL$, is then given by
\begin{equation*}
	\langle \delta \El (y),v\rangle=:\int_{-N}^{N} \SNLL\nabla v\,\d x,
\end{equation*}
where
\begin{equation}\label{LInearization stress tensor}
	\SNLL := \left\{
	\begin{aligned}
		&\sum_{\xi \in \Ac}\sum_{\rho \in \Rc}\rho \Ki \Phia_{\xi,\rho}(y) +\sum_{\xi \in \Ic}\sum_{\rho \in \Rc}\rho \Ki \Phii_{\xi,\rho}(y) \quad x\in \OmeA \cup \OmeI, \\
		&\partial_{F}W(\nabla y) \qquad \qquad \qquad \qquad \qquad \qquad \qquad \qquad \quad  \quad~~ x\in \OmeNL,\\
		&\partial_{F}W_{\text{L}}(\nabla y) \qquad \qquad \qquad \qquad \qquad \qquad \qquad \qquad \qquad~ x\in \OmeL.
	\end{aligned}
	\right.
\end{equation}

We now define the error in linearization as
\begin{equation*}
	\RNLL=\SNLL-\Srfl.
\end{equation*}
We will provide a point-wise estimate of the stress tensor for the QNLL model using the definition of $\WL$.

\begin{theorem}
	Let $y\in \YacNL,x\in[-N,N]$, we have
	\begin{equation}\label{Pointwise NLL stress tensor}
		\vert \RNLL\vert \le \frac{1}{2}M^{(3,0)} \Vert \nabla u(x)\Vert^{2}_{L^{\infty}(\OmeL)}, \ x\in \OmeL.
	\end{equation}
\end{theorem}

\begin{proof}
	We notice that $\RNLL \neq 0$, only for $x\in \OmeL$. We could directly know
	\begin{equation*}
		\partial_{F}\WL(\nabla y) = \Wpf+\Wppf\nabla u.
	\end{equation*}
	After using the Taylor's expansion, we obtain
	\begin{equation*}
		\vert \partial_{F}W_{\text{L}} (\nabla y)-\partial_{F}W(\nabla y)\vert \le \frac{1}{2} M^{(3,0)} \vert \nabla u(x)\vert^{2}\le\frac{1}{2}M^{(3,0)}\Vert \nabla u(x)\Vert^{2}_{L^{\infty}(\OmeL)}.
	\end{equation*}
\end{proof}

Based on the point estimate of $\RNLL$ provided above, we will now present the consistency error estimate for the QNLL model.

\begin{theorem}\label{Linearization consistency error estimate}
	For any $v \in \YacNL$, we have
	\begin{equation}\label{Linearization error estimate}
		\int_{-N}^{N} \RNLL \nabla v\,\d x\le \frac{1}{2} M^{(3,0)}\Vert \nabla u\Vert^{2}_{L^{4}(\OmeL)}\Vert \nabla v\Vert_{L^{2}(\OmeL)}.
	\end{equation} 
\end{theorem}
\begin{proof}
	After using \eqref{Pointwise NLL stress tensor}, we have
	\begin{equation*}
		\int_{\OmeL} \RNLL  \nabla v \,\d x\le \int_{\OmeL} \frac{1}{2}M^{(3,0)}\Vert \nabla u(x)\Vert^{2}_{L^{\infty}(\OmeL)} \vert \nabla v \vert \,\d x.
	\end{equation*}
	We consider \eqref{L-infty to L-2 estimate} and obtain
	\begin{equation*}
		\int_{-N}^{N} \RNLL \nabla v\,\d x\le \frac{1}{2} M^{(3,0)}\Vert \nabla u\Vert^{2}_{L^{4}(\OmeL)}\Vert \nabla v\Vert_{L^{2}(\OmeL)}.
	\end{equation*}
\end{proof}

Finally, by combining Proposition \ref{Coupling consistency error estimate} and Theorem \ref{Linearization consistency error estimate}, and applying the triangle inequality, we provide the consistency error estimate between the QNLL model and the atomistic model.

\begin{theorem}
	For any $y, v \in \Yn$, the consistency error estimate between the QNLL model and the atomistic model is
	\begin{equation}\label{QNLL consistency error estimate}
		\Vert T \Vert_{\Yn^{*}} \le M^{(2,1)}\Vert \nabla^{2} u\Vert_{L^{2}(\bOmeI)} +M^{(2,2)}\Vert \nabla^{3}u \Vert_{L^{2}(\bOmeC)}+M^{(3,2)}\Vert \nabla^{2}u \Vert^{2}_{L^{4}(\bOmeC)}M^{(3,0)}+\Vert \nabla u \Vert^{2}_{L^{4}(\OmeL)}.
	\end{equation}
\end{theorem}


\subsection{Stability}
\label{sec: stability_qnll_ncg}

In this section, we establish two key results regarding the stability of the QNLL model: (i) For uniform deformations, the QNLL model exhibits universal stability, similar to the QNL model; (ii) For non-uniform deformations, the QNLL model progressively stabilizes as the atomistic region expands.

Firstly, we calculate the second variation of the energy functional of Nonlinear-linear elasticity coupling energy defined by \eqref{Nonlinear-linear energy}, for any $v\in\YacNL$, is then given by
\begin{equation}\label{Stab of NL-L}
	\begin{split}
		\langle \delta^{2} \El (y)v,v\rangle =& \sum_{\xi \in \Ac} \sum_{(\rho,\zeta)\in\Rc^{2}}  \Phia_{\xi,\rho\zeta}(y)D_{\rho}v(\xi)D_{\zeta}v(\xi)\\
		&\ +\sum_{\xi \in \Ic} \sum_{(\rho,\zeta)\in\Rc^{2}} \Phii_{\xi,\rho\zeta}(y)D_{\rho}v(\xi)D_{\zeta}v(\xi)\\
		&\ +\int_{\OmeNL}\ppGW(\nabla y) (\nabla v)^{2}\,\d x\\
		&\ +\int_{\OmeL}\ppGWL(\nabla y) (\nabla v)^{2}\,\d x.
	\end{split}
\end{equation}

If we focus on the second variation evaluated at the homogeneous deformation $\yF$, and use the fact $\ppGW(\nabla \yF)=\ppGWL(\nabla \yF)=\Wppf$. Hence, we can obtain
\begin{equation*}
	\int_{\OmeNL}\ppGW(\nabla \yF) (\nabla v)^{2}\,\d x +\int_{\OmeL}\ppGWL(\nabla \yF)  (\nabla v)^{2}\,\d x =\int_{\OmeC}\ppGW(\nabla \yF) (\nabla v)^{2}\,\d x.
\end{equation*}

After a direct calculation, we have
\begin{equation}\label{Stab of QNLL equals to QNL}
	\begin{split}
		\langle \delta^{2} \El (\yF)v,v\rangle 
		=&\sum_{\xi \in \Ac} \sum_{(\rho,\zeta)\in\Rc^{2}}  \Phia_{\xi,\rho\zeta}(\yF)D_{\rho}v(\xi)D_{\zeta}v(\xi)\\
		&\ +\sum_{\xi \in \Ic} \sum_{(\rho,\zeta)\in\Rc^{2}} \Phii_{\xi,\rho\zeta}(\yF)D_{\rho}v(\xi)D_{\zeta}v(\xi)\\
		&\ +\int_{\OmeC}\ppGW(\nabla \yF) (\nabla v)^{2}\,\d x \\
		=&~\langle \delta^{2} \Erfl (\yF)v,v\rangle. 
	\end{split}
\end{equation}

We then define the stability constants (for homogeneous deformations) $\gaaF, \garflF, \ganllF$ as 
\begin{align}
	\label{GammaF for a}  \gaaF &=\inf_{v\in \Ya} \frac{\langle \delta^{2} \Ea (\yF)v,v\rangle}{\Vert \nabla v \Vert_{L^{2}}^{2}},\\
	\label{GammaF for rfl} \garflF &=\inf_{v\in \Ya} \frac{\langle \delta^{2} \Erfl (\yF)v,v\rangle}{\Vert \nabla v \Vert_{L^{2}}^{2}},\\
	\label{GammaF for nll}\ganllF &=\inf_{v\in \Ya} \frac{\langle \delta^{2} \El (\yF)v,v\rangle}{\Vert \nabla v \Vert_{L^{2}}^{2}}.
\end{align}

The QNL method was propose as a ``universally stable method" (cf.~\cite[Theoorem 4.3]{2014_CO_AS_LZ_Stabilization_MMS}). Combining this result with \eqref{Stab of QNLL equals to QNL} we obtain
\begin{equation}\label{Stab constants of three methods result}
	\gaaF = \garflF =\ganllF.
\end{equation}

To make \eqref{Stab of NL-L} precise, we will 
{\it split} the test function $v$ into an atomistic and continuum component, using the following lemma [\cite{2013_ML_CO_AC_Coupling_ACTANUM}, Lemma 7.3]. We leave the proof to Appendix \ref{Appendix section 3}.

\begin{lemma}\label{Pointwise blending lemma}
	Let $\beta \in C^{1,1}(-\infty,\infty)$, with $0\le \beta \le 1$. For each $v\in \Ya$, there exists $\va, \vc \in \Ya$ such that
	\begin{align}
		\label{Pointwise va blending estimate}\vert \sqrt{1-\beta(\xi)} D_{\rho}v(\xi)-D_{\rho}\va(\xi) \vert &\le \vert \rho \vert^{\frac{3}{2}}\Vert \nabla \sqrt{1-\beta}\Vert_{L^{\infty}} \Vert \nabla v \Vert_{L^{2}(\conv(\xi,\xi+\rho))},\\ 
		\label{Pointwise vc blending estimate1}\vert \sqrt{\beta(\xi)} D_{\rho}v(\xi)-D_{\rho}\vc(\xi) \vert &\le \vert \rho \vert^{\frac{3}{2}}\Vert \nabla \sqrt{\beta}\Vert_{L^{\infty}} \Vert \nabla v \Vert_{L^{2}(\conv(\xi,\xi+\rho))},\\ 
		\label{va and vc} \vert \nabla\va\vert^{2} +	\vert \nabla\vc\vert^{2} &= \vert \nabla v \vert^{2}.
	\end{align}
\end{lemma}

	%

	

According to Lemma \ref{Pointwise blending lemma}, we can obtain the conclusion about the stability of the QNLL model in the case of non-uniform deformations.

\begin{theorem}\label{Stability}
	Let $y\in\Ya$ satisfy the strong stability condition \eqref{All-Atomistic strong local minimizer} and suppose that there exists $\ganllF >0$ such that
	\begin{equation}\label{Uniform solution of NL-L stab assumption}
		\langle \delta^{2}\El (\yF)v,v\rangle \ge \ganllF \Vert \nabla v\Vert^{2}_{L^{2}(-N,N)} \text{ for all } v \in \Ua.
	\end{equation}
	
	Then
	\begin{equation}\label{Result of stability}
		\begin{split}
			\langle \delta^{2}\El (y)v,v\rangle &\ge \min(c_{0},\ganllF)\Vert \nabla v \Vert_{L^{2}(-N,N)}^{2}\\
			&\ - 2M^{(2,\frac{1}{2})}K^{-1} \Vert \nabla v \Vert_{L^{2}(-N,N)}^{2} - \CDH M^{(3,0)}K^{-\alpha} \Vert \nabla v \Vert_{L^{2}(-L,L)}^{2}, \quad \text{for all } v \in \Ua.
		\end{split}
	\end{equation}
\end{theorem}

\begin{proof}
	According to Lemma \ref{Pointwise blending lemma}, let $K^{'}:=\lfloor K/2\rfloor <K$, and let
	\begin{equation*}
		\beta (x):=\left\{
		\begin{aligned}
			&0, &-K^{'}\le x \le K^{'},\\
			&\hat{\beta}(\frac{x+K^{'}}{K^{'}-K}), &-K\le x \le -K^{'},\\
			&\hat{\beta} (\frac{x-K^{'}}{K-K^{'}}), &K^{'} \le x \le K,\\
			&1, &-N\le x \le -K \text{ or } K\le x\le N.
		\end{aligned}
		\right.
	\end{equation*}
	where $\hat{\beta}(s)=3s^{3}-2s^{2}$. We know from \cite[Section 8.3]{2013_ML_CO_AC_Coupling_ACTANUM} that 
	\begin{equation}\label{Property of beta}
		\Vert \nabla \sqrt{\beta} \Vert_{L^{\infty}} + \Vert \nabla \sqrt{1-\beta} \Vert_{L^{\infty}} \leq 	C_{\beta} K^{-1}.
	\end{equation}
	
	We can now write
	\begin{subequations}
		\begin{align}
			\langle \delta^{2} \El (y)v,v\rangle &= \sum_{\xi=-N}^{N} \sum_{(\rho,\zeta)\in\Rc^{2}} \Phia_{\xi,\rho\zeta} (y)\big(1-\beta(\xi)\big) D_{\rho}v(\xi) D_{\zeta}v(\xi) \nonumber\\
			&\quad +\sum_{\xi=-N}^{N} \sum_{(\rho,\zeta)\in\Rc^{2}} \PhiNLL_{\xi,\rho\zeta} (y)\big(\beta(\xi)\big) D_{\rho}v(\xi) D_{\zeta}v(\xi) \nonumber\\
			&= 
			\label{All atomistic stab}
			\sum_{\xi=-N}^{N} \sum_{(\rho,\zeta)\in\Rc^{2}} \Phia_{\xi,\rho\zeta} (y)\big(1-\beta(\xi)\big) D_{\rho}v(\xi) D_{\zeta}v(\xi)\\
			\label{NL-L atomistic stab} 
			&\quad +\sum_{\xi \in \Ac} \sum_{(\rho,\zeta)\in\Rc^{2}} \Phia_{\xi,\rho\zeta} (y)\big(\beta(\xi)\big)D_{\rho}v(\xi)D_{\zeta}v(\xi)\\
			\label{NL-L interaction stab} 
			&\quad +\sum_{\xi \in \Ic} \sum_{(\rho,\zeta)\in\Rc^{2}} \Phii_{\xi,\rho\zeta} (y)\big(\beta(\xi)\big)D_{\rho}v(\xi)D_{\zeta}v(\xi)\\
			\label{NL-L nonlinear stab} 
			&\quad + \int_{\OmeNL} \ppGW (\nabla y)\big(\beta (x)\big) (\nabla v)^{2}\,\d x\\
			\label{NL-L linear stab} 
			&\quad +\int_{\OmeL}  \ppGWL (\nabla y)\big(\beta (x)\big)(\nabla v)^{2}\,\d x.
		\end{align}
	\end{subequations}
	where we also use the fact that according to our definition of $\beta$, the first sum ranges only over those sites where $\PhiNLL_{\xi} = \Phia_{\xi}$.
	
	We apply estimate \eqref{Pointwise va blending estimate} to \eqref{All atomistic stab}, and we obtain
	\begin{equation}\label{Atomistic va stab result}
		\begin{split}
			\sum_{\xi=-N}^{N} \sum_{(\rho,\zeta)\in\Rc^{2}} \Phia_{\xi,\rho\zeta} (y)\big(1-\beta(\xi)\big) D_{\rho}v(\xi) D_{\zeta}v(\xi) &\ge \langle \delta^{2} \Ea (y)\va,\va\rangle 
			- 2M^{(2,\frac{1}{2})}K^{-1}\Vert \nabla v \Vert_{L^{2}(-N,N)}^{2}\\
			&\ge c_{0} \Vert \nabla \va \Vert_{L^{2}(-N,N)}^{2}-2M^{(2,\frac{1}{2})}K^{-1}\Vert \nabla v \Vert_{L^{2}(-N,N)}^{2}.
		\end{split}
	\end{equation}
	
	We apply the estimate \eqref{Pointwise vc blending estimate1} to \eqref{NL-L atomistic stab} to get 
	\begin{equation}\label{NL-L atomistic stab 1}
		\begin{split}
			\sum_{\xi \in \Ac} \sum_{(\rho,\zeta)\in\Rc^{2}} \Phia_{\xi,\rho\zeta} (y)\big(\beta(\xi)\big)D_{\rho}v(\xi)D_{\zeta}v(\xi) \ge& \sum_{\xi \in \Ac} \sum_{(\rho,\zeta)\in\Rc^{2}} \Phia_{\xi,\rho\zeta} (y)D_{\rho}\vc(\xi)D_{\zeta}\vc(\xi) \\
			&\ - 2M^{(2,\frac{1}{2})} K^{-1} \Vert \nabla v \Vert_{L^{2}(\OmeA)}^{2}.
		\end{split}
	\end{equation}
	
	By the definition of $\vc$, we notice that for $x \in[-K^{'},K^{'}]$, $\nabla \vc = 0$. After using Taylor's expansion at $\yF$ and $\DH$ assumption, we obtain
	\begin{equation}\label{NL-L atomistic stab 2}
		\begin{split}
			\sum_{\xi \in \Ac} \sum_{(\rho,\zeta)\in\Rc^{2}} \Phia_{\xi,\rho\zeta} (y)D_{\rho}\vc(\xi)D_{\zeta}\vc(\xi)\ge& \sum_{\xi \in \Ac} \sum_{(\rho,\zeta)\in\Rc^{2}} \Phia_{\xi,\rho\zeta} (\yF)D_{\rho}\vc(\xi)D_{\zeta}\vc(\xi)\\
			&\ -2^{\alpha}\CDH M^{(3,0)}  (K)^{-\alpha} \Vert\nabla \vc \Vert_{L^{2}(\OmeA)}^{2}.
		\end{split}
	\end{equation}
	
	Combining \eqref{NL-L atomistic stab 1} with \eqref{NL-L atomistic stab 2}, we can obtain
	\begin{equation}\label{NL-L atomistic stab result}
		\begin{split}
			\sum_{\xi \in \Ac} \sum_{(\rho,\zeta)\in\Rc^{2}} \Phia_{\xi,\rho\zeta} (y)\big(\beta(\xi)\big)D_{\rho}v(\xi)D_{\zeta}v(\xi) &\ge \sum_{\xi \in \Ac} \sum_{(\rho,\zeta)\in\Rc^{2}} \Phia_{\xi,\rho\zeta} (\yF)D_{\rho}\vc(\xi)D_{\zeta}\vc(\xi) \\
			&\ - 2M^{(2,\frac{1}{2})} K^{-1} \Vert \nabla v \Vert_{L^{2}(\OmeA)}^{2}-2^{\alpha}\CDH M^{(3,0)}  K^{-\alpha}\Vert\nabla \vc \Vert_{L^{2}(\OmeA)}^{2}.
		\end{split}
	\end{equation}
	
	Similarly, for the term \eqref{NL-L interaction stab}, we have
	\begin{equation}\label{NL-L interaction stab result}
		\begin{split}
			\sum_{\xi \in \Ic} \sum_{(\rho,\zeta)\in\Rc^{2}} \Phii_{\xi,\rho\zeta} (y)\big(\beta(\xi)\big)D_{\rho}v(\xi)D_{\zeta}v(\xi) &\ge \sum_{\xi \in \Ic} \sum_{(\rho,\zeta)\in\Rc^{2}} \Phii_{\xi,\rho\zeta} (\yF)D_{\rho}\vc(\xi)D_{\zeta}\vc(\xi) \\
			&\ - M^{(2,\frac{1}{2})} K^{-1} \Vert \nabla v \Vert_{L^{2}(\OmeI)}^{2}-\CDH M^{(3,0)}  K^{-\alpha} \Vert\nabla \vc \Vert_{L^{2}(\OmeI)}^{2}.
		\end{split}
	\end{equation}
	
	After considering the definition of $\vc$ and assumption $\DH$, for~\eqref{NL-L nonlinear stab}, we have 
	\begin{equation}\label{NL-L nonlinear stab result}
		\int_{\OmeNL} \ppGW (\nabla y)\big(\beta(x)\big)(\nabla v)^{2}\,\d x \ge 	\int_{\OmeNL} \ppGW (\nabla \yF)(\nabla\vc)^{2}\,\d x - 2\CDH M^{(3,0)}  K^{-\alpha} \Vert \nabla \vc \Vert_{L^{2}(\OmeNL)}^{2}.
	\end{equation}

	We use the fact that $\ppGWL (\nabla y) = \Wppf = \ppGWL (\nabla \yF)$ again, and we can rewrite \eqref{NL-L linear stab} as
	\begin{equation}\label{NL-L linear stab result}
		\int_{\OmeL}\ppGWL (\nabla y)\big(\beta(x)\big) (\nabla v)^{2} \,\d x = \int_{\OmeL}\ppGWL (\nabla \yF) (\nabla \vc)^{2} \,\d x.
	\end{equation}
	
		%
		%
		%
		%

	By considering the definition of $\langle \delta^{2} \El (\yF) \vc,\vc \rangle$, \eqref{va and vc} and \eqref{Property of beta}, we conclude that 
	\begin{equation*}
		\langle \delta^{2} \El (y)v,v\rangle \ge \langle \delta^{2} \El (\yF) \vc,\vc \rangle + \langle \delta^{2}\Ea(y)\va,\va \rangle.
	\end{equation*}
	
\end{proof}


\subsection{A priori existence and error estimate}
\label{sec: priori_qnll_ncg}

In this section, based on the consistency error estimate \eqref{QNLL consistency error estimate} and stability analysis \eqref{Result of stability} of the QNLL model, we will provide a priori error analysis for the QNLL model using the inverse function theorem.

\begin{theorem}\label{Priori of NCG}
	Let $\yai \in \Ya$ be a strongly stable atomistic solution satisfying \eqref{All-Atomistic strong local minimizer} and $\DH$. Consider the QNLL problem \eqref{Nonlinear-linear solution condition}, supposing, moreover, that $\El$ is stable in the reference state Theorem \ref{Stability}. Then there exists $K_0$ such that, for all $K \ge K_0$, \eqref{Nonlinear-linear solution condition} has a locally unique, strongly stable solution $\ynll$ which satisfies
	\begin{equation}
		\begin{aligned}
			\Vert \nabla\yai - \nabla \ynll \Vert_{L^2} \lesssim 8&M^{(3,0)}(\Vert \nabla^{2} u\Vert_{L^{2}(\bOmeI)} +\Vert \nabla^{3}u \Vert_{L^{2}(\bOmeC)}+\Vert \nabla^{2}u \Vert^{2}_{L^{4}(\bOmeC)}\\
			&+ \Vert \nabla u \Vert^{2}_{L^{4}(\OmeL)}+N^{\frac{1}{2}-\alpha})/\big(\min(c_{0},\ganllF)\big)^2.
		\end{aligned}
	\end{equation}
	
\end{theorem}
\begin{proof}
	We will first provide the a priori error estimate for $\Vert \nabla \ynll - \nabla \ya\Vert_{L^2}$ using the quantitative inverse function theorem, with
	\begin{equation*}
		\Ghc( \ya):=\delta \El( \ya) - \langle f,\cdot\rangle_{N}.
	\end{equation*}
	We first apply that the scaling condition implies a Lipschitz bound for $\delta^{2}\El$, 
	\begin{equation}\label{scaling assumption}
		\Vert \delta^2 \El (y) - \delta^2 \El(v)\Vert_{\mathcal{L}(\Ya,\Ya^{*})}\le M \Vert \nabla y-\nabla v\Vert_{L^{\infty}}, \quad \text{for all } v \in \Ua,
	\end{equation}
	where $M \lesssim M^{(3,0)}$. Since $\Vert \cdot \Vert_{\infty} \lesssim \Vert \cdot \Vert_{L^2}$, we can also replace the $L^\infty$- norm on the right-hand side with the $L^2$-norm.
	The residual estimate \eqref{QNLL consistency error estimate} gives
	\begin{equation*}
		\Vert \Ghc( \ya)\Vert_{\Yn}\lesssim 
		M^{(2,1)}\Vert \nabla^{2} u\Vert_{L^{2}(\bOmeI)} +M^{(2,2)}\Vert \nabla^{3}u \Vert_{L^{2}(\bOmeC)}+M^{(3,2)}\Vert \nabla^{2}u \Vert^{2}_{L^{4}(\bOmeC)}+ M^{(3,0)}\Vert \nabla u \Vert^{2}_{L^{4}(\OmeL)}
	\end{equation*}
	From Theorem \ref{Stability} we obtain that
	\begin{equation*}
		\langle \delta^{2}\El (\ya)v,v\rangle \ge (\min(c_{0},\ganllF)-CK^{-\min(1,\alpha)})\Vert \nabla v \Vert_{L^{2}}^{2}.
	\end{equation*}
	Let $\gamma:=\frac{1}{2}\min(c_{0},\ganllF)$. Applying the Lipschitz bound \eqref{scaling assumption}, and we obtain
	\begin{equation*}
		\langle \delta^{2}\El (\ya)v,v\rangle \ge (2\gamma-CK^{-\min(1,\alpha)}-CK^{-1/2-\alpha})\Vert \nabla v \Vert_{L^{2}}^{2}.
	\end{equation*}
	Hence, for $K$ sufficiently large, we obtain that
	\begin{equation*}
		\langle \delta\Ghc (\ya)v,v\rangle \ge \gamma\Vert \nabla v \Vert_{L^{2}}^{2}, \quad \text{for all }v\in\Un.
	\end{equation*}
	Thus, we deduce the existence of $\ynll$ satisfying $\Ghc(\ynll)=0$. The error estimate implies
	\begin{equation*}
		\begin{aligned}
			\Vert \nabla \ynll - \nabla \ya\Vert_{L^2}&\lesssim \frac{2M\eta}{\gamma^2} \\
			&\lesssim  2M^{(3,0)}(\Vert \nabla^{2} u\Vert_{L^{2}(\bOmeI)} +\Vert \nabla^{3}u \Vert_{L^{2}(\bOmeC)}+\Vert \nabla^{2}u \Vert^{2}_{L^{4}(\bOmeC)}
			+ \Vert \nabla u \Vert^{2}_{L^{4}(\OmeL)})/\gamma^2.
		\end{aligned}
	\end{equation*}
	Finally, by using the triangle inequality and truncation error \eqref{Truncation error}, we yield the stated result.
	
\end{proof}

\subsection{Discussion of the (quasi-)optimal choice of the length of regions}
\label{sec: balance_of_qnll_ncg_model}

In this subsection, we will discuss how to achieve the quasi-optimal choice of the lengths for nonlinear continuum region, linear continuum region, and computational region to obtain quasi-optimal convergence order for the QNLL model.

\subsubsection{The quasi-optimal choice of $L$}
\label{sec: choice_of_L_ncg}

We aim to balance the lengths of nonlinear continuum region and linear continuum region by incorporating coupling error estimates \eqref{Interface region stress tensor}, \eqref{Continuum region stress tensor} and linearization error estimate \eqref{Linearization error estimate}, under the assumption of $\DH$.

First, we use the $\DH$ assumption to obtain the convergence order of coupling error estimates \eqref{Interface region stress tensor}, \eqref{Continuum region stress tensor} concerning the length of atomistic region and nonlinear continuum region :
\begin{equation*}
	\begin{split}
		M^{(2,1)}\Vert \nabla^{2} u&\Vert_{L^{2}(\bOmeI)} +M^{(2,2)}\Vert \nabla^{3}u \Vert_{L^{2}(\bOmeC)}+M^{(3,2)}\Vert \nabla^{2}u \Vert^{2}_{L^{4}(\bOmeC)}M^{(3,0)}\\
		&\lesssim \CDH M^{(2,1)}K^{-\alpha -1}+\CDH M^{(2,2)}\bar{K}^{-\alpha-\frac{3}{2}} +\CDH^{2}M^{(3,2)}\bar{K}^{-2\alpha -\frac{3}{2}}.
	\end{split}
\end{equation*}
Here we need to note the fact that $K + 2 = \bar{K}$, so we can assume $K\approx \bar{K}$. The lowest-order term among them is $\Vert \nabla^{2} u\Vert_{L^{2}(\bOmeI)} \lesssim K^{-\alpha-1}(\bar{K}^{-\alpha-1})$.

Next, we will similarly apply the $\DH$ assumption to the linearization error estimate \eqref{Linearization error estimate} to obtain its convergence order with respect to the length of the linear continuum region :

\begin{equation*}
	M^{(3,0)}\Vert \nabla u \Vert^{2}_{L^{4}(\OmeL)} \lesssim \CDH^{2}M^{(3,0)}L^{-2\alpha+\frac{1}{2}}.
\end{equation*}

The lowest-order term of $L$ is $\Vert \nabla u \Vert^{2}_{L^{4}(\OmeL)} \lesssim L^{-2\alpha+\frac{1}{2}}$. We balance this term with $\Vert \nabla^{2} u\Vert_{L^{2}(\bOmeI)} \lesssim \bar{K}^{-\alpha-1}$ to get (by noticing the fact that $\bar{K}\le L$)
\begin{align}
	L &\lesssim \bar{K}^{\frac{1}{2}+\frac{5}{8\alpha-2}}(\frac{1}{2}<\alpha<\frac{3}{2}) \label{Balance of L NCG 1},\\
	L &\approx \bar{K}(\alpha \ge \frac{3}{2})\label{Balance of L NCG 2}.
\end{align}

Because through balancing we have made the orders of the lowest order terms of $\bar{K}$ and $L$ equal, for simplicity in this section, we will uniformly use linearization error $L^{-2\alpha+\frac{1}{2}}$ to represent the lowest-order term of coupling error and linearization error.

\subsubsection{The quasi-optimal choice of $N$}
\label{sec: choice_of_N_ncg}

After balancing the lengths of nonlinear continuum region and linear continuum region, we will now balance the computational domain length, which will follow the principles:

\begin{enumerate}
	\item We should ensure that the truncation error term $N^{\frac{1}{2}-\alpha}$ do not dominate among the various types of errors after balancing the length of the computational domain;
	
	\item We choose the length of the computational domain as small as possible for computational simplicity.
\end{enumerate}

According to the first principle mentioned above, we understand that the truncation error $N^{\frac{1}{2}-\alpha}$ must be balanced against one of the terms of coupling error or linearization error (or higher-order terms). According to the second principle, to select the computational domain length $N$ as small as possible, we must balance it against the lowest-order term of coupling error or linearization error (balancing against higher-order terms would need a longer computational domain length).

After balancing the lengths of nonlinear continuum region and linear continuum region, the lowest-order term is $L^{\frac{1}{2}-2\alpha}$, we should choose $N$ such that
\begin{equation*}
	L^{\frac{1}{2}-2\alpha}\approx N^{\frac{1}{2}-\alpha}, \quad \text{that is}, \ N\approx L^{\frac{2\alpha-1/2}{\alpha-1/2}}.
\end{equation*}

\subsection{Numerical validation}
\label{sec: experiments_qnll_ncg}

In this section, we present numerical experiments to illustrate the result of our analysis. With slight adjustments, the problem we consider here is a typical testing case in one dimension. We fix $F=1$ and let $V$ be the site energy given by the embedded atom method (EAM)~\cite{1984_Daw_Baskes_EAM_PRB}:
\begin{equation}\label{EAM of numerical experiments}
	V\big(Dy(l)\big)=\frac{1}{2} \sum_{i\in\{1,2\};j\in\{-1,-2\}}\big(\phi(D_{i}y_{l})+\phi(-D_{j}y_{l})\big)+\widetilde{F}\left(\sum_{i\in\{1,2\};j\in\{-1,-2\}}\big[\psi(D_{i}y_{l})+\psi(-D_{j}y_{l})\big]\right),
\end{equation}
where $\phi(r) = \exp\big(-2a(r-1)\big)-2\exp\big(-a(r-1)\big), \psi(r) = \exp(-br)$, and $\tilde{F}(\rho) = c[(\rho-\rho_{0})^{2}+(\rho-\rho_{0})^{4}]$, with the parameter $a=4.4, b=3, c=5,\rho_{0} =2\exp(-b)$.

We fix an exact solution
\begin{equation}\label{External force of numerical experiments}
	\ya(\xi):=F\xi + \frac{1}{10}(1+\xi^2)^{\alpha/2}\xi,
\end{equation}
and compute the external forces $f(\xi)$ to be the equal to the internal forces under the deformation $\ya$. The parameter $\alpha$ is a prescribed decay exponent. One may readily check that this solution and the associated external forces satisfy the decay hypothesis $\DH$.

We will demonstrate the method of controlling the length of non-linear continuum region in the QNLL model to achieve quasi-optimal convergence order, as introduced in Section~\ref{sec: balance_of_qnll_ncg_model}. We will conduct numerical experiments with the atomistic model length of 100,000 atoms. We set energy functional and external force to \eqref{EAM of numerical experiments} and \eqref{External force of numerical experiments}. The experiments will be carried out for $\alpha$ values of $1.2, 1.5$ and $1.8$.


\begin{figure}[h!]
	\centering
	\subfloat[$\alpha = 1.8$]{
		\includegraphics[width=0.3\textwidth]{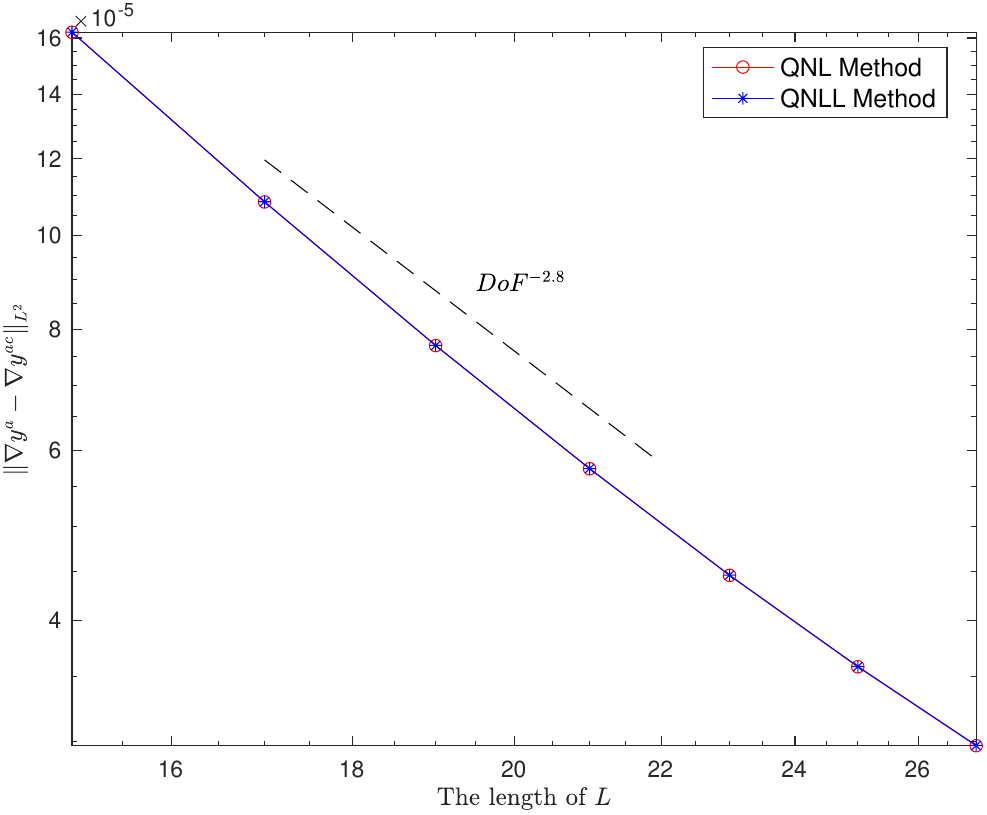}
		\label{fig: convergence_QNL_QNLL_alpha18_NCG}
	}
	\subfloat[$\alpha = 1.5$]{
		\includegraphics[width=0.3\textwidth]{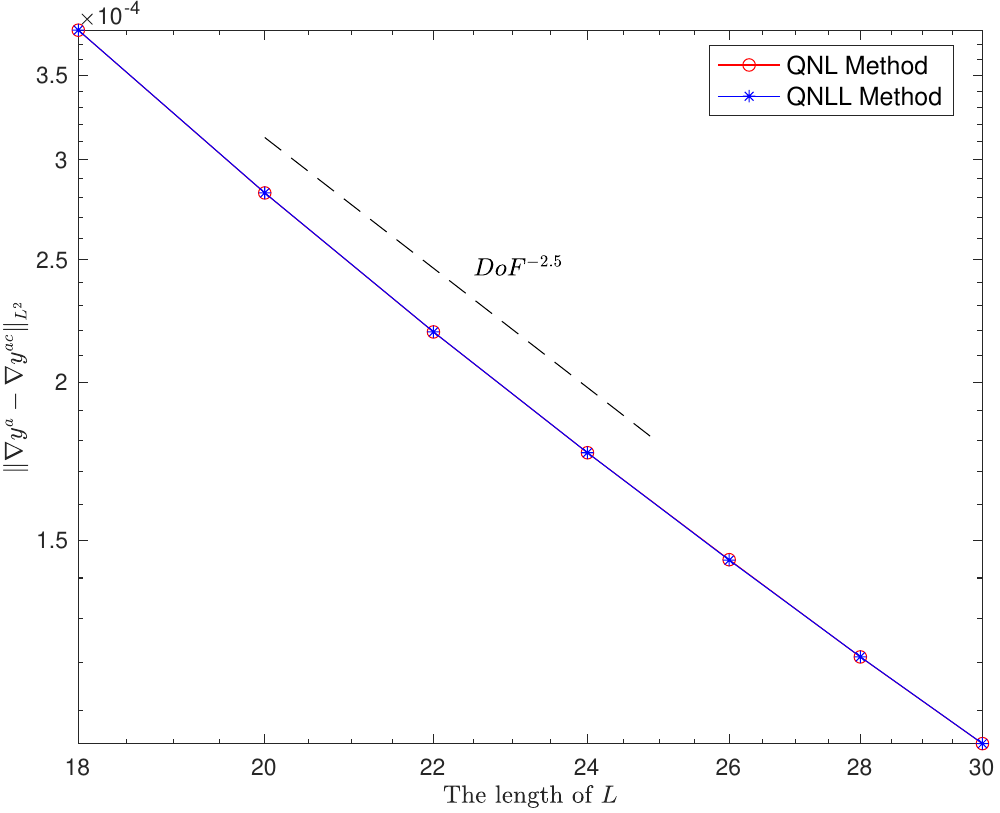}
		\label{fig: convergence_QNL_QNLL_alpha15_NCG}
	}
	\subfloat[$\alpha = 1.2$]{
		\includegraphics[width=0.3\textwidth]{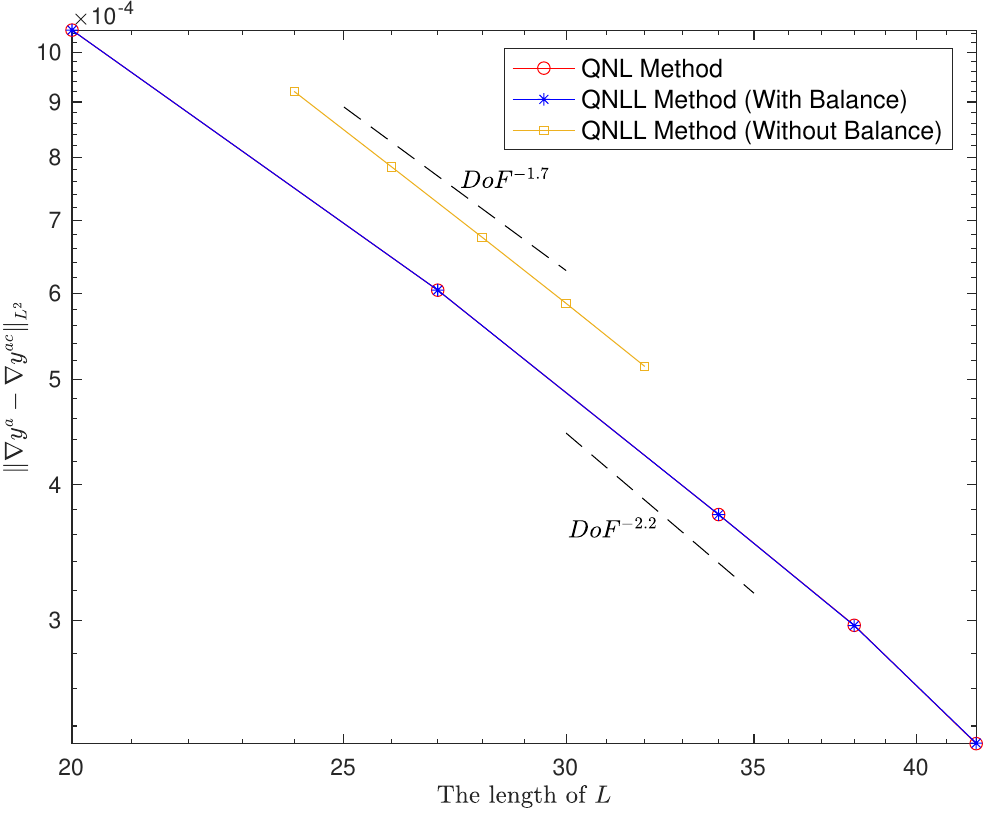}
		\label{fig: convergence_QNL_QNLL_alpha12_NCG}
	}
	\caption{The convergence order of QNL and QNLL method with different $\alpha$ (without coarse-graining)}
	\label{fig: convergence_QNL_QNLL_NCG}
\end{figure}

Firstly, let us consider the experiment with alpha set to $1.8$: In this case, according to \eqref{Balance of L NCG 2}, by setting the length of the nonlinear continuum region to a few atoms ($\bar{K} \approx L$), the convergence order of the QNLL method matches that of the QNL method. In the Figure below, the $x$-axis represents the length of $L$, while the $y$-axis shows the absolute error $\Vert \nabla \yai - \nabla y^{\text{ac}} \Vert_{L^{2}}~\text{(ac} = \text{QNL, QNLL)}$ between the reference atomistic solution $\yai$ and the AC solutions $y^{\text{ac}}$. It can be observed that the two convergence order lines in the graph nearly overlap, indicating that the difference between the two AC solutions $\Vert \nabla y^{\text{QNL}} - \nabla y^{\text{QNLL}}\Vert$ is between $10^{-6} $and $10^{-7}$.


When $\alpha = 1.5$, the results are similar to when $\alpha = 1.8$. The figure above compares the convergence order of the QNLL method and the QNL method. The information represented on the axes is the same as in Figure \ref{fig: convergence_QNL_QNLL_alpha12_NCG}. We observe a similar outcome to Figure \ref{fig: convergence_QNL_QNLL_alpha18_NCG}, where the convergence lines of the QNLL method closely overlap with those of the QNL method.


Furthermore, we will now consider the case where $\alpha=1.2$. In this setting, according to \eqref{Balance of L NCG 1} and \eqref{Balance of L NCG 2}, there are two mesh generation schemes for the QNLL method:
\begin{enumerate}
	\item In the first scheme, we focus on the accuracy of the QNLL method. According to \eqref{Balance of L NCG 1}, we precisely balance the atomistic region, nonlinear continuum region, linear continuum region, and the total length of the computational domain to achieve convergence order identical to those of the QNL method.
	
	\item In the second scheme, we prioritize the computational efficiency of the QNLL method. Therefore, after balancing the lengths of the atomistic region and the total length of the computational domain, we minimize the length of the nonlinear continuum region as much as possible, even down to just a few atoms.
\end{enumerate}


In the figure below, we represent the first mesh generation scheme with red dashed squares for the QNLL method, and the second generation scheme with blue dashed squares. To demonstrate the accuracy of the QNLL method, the QNL method also adopts the first mesh generation scheme, depicted in the figure with red dashed star symbols. The information represented on the axes is the same as in Figure \ref{fig: convergence_QNL_QNLL_alpha18_NCG}. We observe that, after balancing the lengths of the atomisticc region, nonlinear continuum region, linear continuum region, and the total length of the computational domain, the absolute errors and convergence order obtained by the QNLL method are consistent with those of the QNL method. However, after reducing the length of the nonlinear continuum region in pursuit of computational efficiency, there is a noticeable increase in absolute errors and a decrease in convergence speed.

Next, to demonstrate the computational efficiency of the QNLL method where $\alpha = 1.2$, we test the variation in computation time by progressively increasing the degrees of freedom of the nonlinear continuum region $\Nnl$ of the QNLL method, while keeping the finite element mesh fixed, meaning the continuum region remains unchanged. The results are as shown in the table below: the first column lists the method names, with parentheses indicating the proportion of the degrees of the freedom of nonlinear continuum region $\Nl$ to that of the total continuum region $\Nc$; the second column denotes the total degrees of freedom of the mesh and the third column records the ratio of the computing time of the QNLL method to the computing time of the QNL method on a device with an M1 CPU and 16 GB of RAM:


\begin{table}
	\centering
	\renewcommand{\arraystretch}{1.5} 
	\begin{tabular}{|c|c|} 
		\hline 
		Method ($\Nnl/\Nc$) & The ratio of the computing time\\ 
		\hline 
		QNLL ($24.99\%$) & $66.68\%$ \\ 
		QNLL ($40.97\%$) & $77.66\%$ \\ 
		QNLL ($74.96\%$) & $88.88\%$ \\ 
		QNL ($100\%$) & $100\%$ \\ 
		\hline 
	\end{tabular}
	\caption{The computing time (without coarse graining) of QNL and QNLL method without coarse-graining ($\alpha = 1.2$), with Degree of Freedom (DoF) set to 500000 for all methods.}
	\label{tab:computing time alpha12 NCG}
\end{table}

Table \ref{tab:computing time alpha12 NCG} clearly shows that, with fixed lengths of the Atomistic and Continuum regions, the computing time increases significantly as the proportion of nonlinear elements in the Continuum region rises. However, in practical applications, the proportion of nonlinear elements will be lower (below 5$\%$) according to the balancing method described in Section \ref{sec: balance_of_qnll_ncg_model}. The ratio of the difference between the absolute errors of the QNLL method and the QNL method to the absolute errors of the QNL method: $( \Vert \nabla \yai - \nabla y^{\text{QNLL}} \Vert_{L^{2}} - \Vert \nabla \yai - \nabla y^{\text{QNL}} \Vert_{L^{2}}) / \Vert \nabla \yai - \nabla y^{\text{QNL}} \Vert_{L^{2}}$ is in a narrow range. Here, the ratio, as defined above, is within the range of $10^{-5}$ to $10^{-6}$. This indicates that the QNLL method maintains high accuracy while still offering computational efficiency advantages.
	
	\section{QNLL Method with Coarse-Graining}
\label{sec: qnll_cg}

The QNLL method we analyze in Section~\ref{sec: anal_qnll_ncg} is not a computable scheme since it considers every atom in the computational domain as a degree of freedom and the computational cost gets high fast as the computational domain or $N$ goes large. Therefore, as a common practice of the a/c method, we need to coarse grain the continuum region to reduce the number of degrees of freedom and consequently the computational cost. 

In this section, we follow the same analysis framework as Section~\ref{sec: anal_qnll_ncg}. However, the difference is that in the consistency error part, we incorporate the error introduced by coarse graining. First, we give the formulation and analysis of the coarse-grained QNLL method. Then, we focus on the balancing of the atomistic, nonlinear, and linear regions so that the (quasi-)optimal convergence of the QNLL method, comparable to that of the QNL method, is achieved. Finally, we present several numerical experiments to demonstrate that the QNLL method, with proper balance of the different regions, retains the same level of accuracy as the QNL method while substantially lowering the computational cost measured by CPU time.

\subsection{Coarse-graining and analysis of the QNLL method}
\label{sec: anal_qnll_cg}


Let $\ThNL = \{T_j\}_{j = 1}^{J} := \big\{[v_{j-1},v_{j}]\ | \ j=1,\dots,J\big\}$ be a regular partition of the computational domain $[-N,N]$ into closed intervals or elements $T_j$. We assume that the vertices of the partition are all at atoms or lattice sites and are denote  by $\NhNL := \{v_{0},\dots,v_J \} \subset \Z_{+}$ (which are often termed as rep-atoms in the language of the quasicontinuum method). We define the coarse-grained space of displacements by
\begin{equation}\label{UhNL space}
	\UhNL:=\{u_{h}\in\Un \ | \ u_{h}(-N)=u_{h}(N)=0 \text{ and }u_{h} \text{ is piecewise affine with respect to }\ThNL\},
\end{equation}
and subsequently the admissible set of deformations by
\begin{equation}\label{YhNL space}
	\YhNL:=\{y\ | \  y = Fx+u_{h}, \ u_{h}\in\UhNL\}.
\end{equation}

We define interpolation operator $I_h: \Yn \rightarrow \YhNL$ by $\IhNL v(\zeta) =v(\zeta), \forall \zeta \in \NhNL$ and $\IhNL v \in \text{P1} (\ThNL)$, which is the piecewise affine nodal interpolation with respect to $\ThNL$. We firstly introduce the following proposition obtained from Poincare's inequality for future usage.

\begin{proposition}
	Let $T\in \ThNL , \ T \subset [-N,-\bar{K}]\cup[\bar{K},N]$ and $u\in \mathscr{U}$. Then
	\begin{equation*}
		\Vert \nabla u-\nabla \IhNL u\Vert_{L^{2}(T)}\lesssim h_{T} \Vert \nabla^{2} u\Vert_{L^{2}(T)}.
	\end{equation*}
\end{proposition}



For each $T\in \ThNL$, we let $h_{T}:=\text{diam} (T)$. Thus, for $f,g \in \Un$, we define 
\begin{equation*}
	\langle f,g\rangle_{h}:= \int_{-N}^{N} \IhNL (f\cdot g)\,\d x=\sum_{j=1}^{J} \frac{1}{2} h_{T}\big\{f(v_{j-1})\cdot g(v_{j-1}) + f(v_{j})\cdot g(v_{j})\big\}.
\end{equation*}

The coarse-grained QNLL model we aim to solve is the following: 
\begin{equation}
	\label{Yh solution}
	\yh \in \argm \{\El (y_{h})-\langle f,y_{h} \rangle_{h}\ | \ y_{h}\in \YhNL\}.
\end{equation}

%
%
%


\subsubsection{Coarsening error of the internal forces}
\label{sec: internal_forces_qnll_cg}
The first variation of the continuum energy contribution $\int_{\OmeC} W(\nabla y)\text{d}x$ is given by
\begin{equation*}
	\int_{\OmeNL} \partial_{F} W(\nabla y)\nabla v \text{d}x+\int_{\OmeL} \partial_{F} W_{\text{L}}(\nabla y)\nabla v \,\textrm{d}x.
\end{equation*}
Elements of $\UhNL$ are defined pointwise, but give rise to lattice functions through point evaluation. Since finite element nodes lie on lattice sites, this is compatible with our interpolation of lattice functions.

The following lemma estimates the error contribution from this operator induced by finite element coarsening and reduction to a finite domain.

\begin{theorem}\label{Internal forces of continuum region}
	Let $u \in \mathscr{U}$ satisfy $(\mathbf{DH})$, and $0<\bar{K} <L \le N/2$. Then
	\begin{equation*}
		\begin{aligned}
			\Bigg \vert \int_{\OmeNL} \big(\partial_{F} W(\nabla \IhNL y) &- \partial_{F} W(\nabla y)\big) \nabla v_{h} \, \text{d} x + \int_{\OmeL} \big(\partial_{F} \WL (\nabla \IhNL y) - \partial_{F} \WL (\nabla y)\big) \nabla v_{h} \,\text{d} x \Bigg \vert \\
			&\lesssim M^{(2,0)}\Vert h\nabla^{2}u\Vert^{2}_{L^{4}(\OmeNL)} \Vert \nabla v_{h} \Vert_{L^{2}}, \quad \text{for all}\  v_{h} \in \text{P1}(\mathcal{T}_{h}).\\
		\end{aligned}
	\end{equation*}
\end{theorem}

\begin{proof}

	Firstly, we calculate the first integral on nonlinear region. After using the fact that $\int_{T}\nabla \IhNL u\,\d x=\int_{T} \nabla u\,\d x$, for any $T\in \OmeNL$, we have
	\begin{equation*}
		\begin{aligned}
			\Bigg \vert \int_{T} \partial_{F} W(\nabla \IhNL y) - \partial_{F} W(\nabla y) \,\d x \Bigg \vert &\le \Bigg \vert \int_{T}\partial^{2}_{F}W(\nabla\IhNL y)(\nabla\IhNL u -\nabla u)\,\d x \Bigg \vert\\
			&\quad + M^{(3,0)}\int_{T}\vert \nabla \IhNL u-\nabla u\vert^{2}\,\d x\\
			&\lesssim \Vert \nabla \IhNL u -\nabla u\Vert^{2}_{L^{2}(T)}\lesssim \Vert h\nabla^{2} u\Vert^{2}_{L^{2}(T)}.
		\end{aligned}
	\end{equation*}
	
	Summing over $T\subset \ThNL$ and again applying H$\ddot{\text{o}}$lder‘s inequality yields
	\begin{equation*}
		\sum_{T \in \ThNL} \Vert h\nabla^{2} u\Vert^{2}_{L^{2}(T)}\Vert \nabla v_{h} \Vert_{L^{2}(T)}
		\le \Vert h\nabla^{2} u\Vert^{2}_{L^{4}(\OmeNL)}\Vert \nabla v_{h} \Vert_{L^{2}(\OmeNL)}.
	\end{equation*}

	Next we focus on the linear region, after using the definition of $\WL(F)$, and we have
	\begin{equation*}
		\partial_{F} W_{\text{L}}(\nabla I_{h} u) - \partial_{F} W_{\text{L}}(\nabla u) =\Wppf(\nabla I_{h}u - \nabla u).
	\end{equation*}
	Moreover, using the fact that $\int_{T} \nabla I_{h}u \text{d}x = \int_{T} \nabla u\text{d}x$, since $\nabla v_{h} $ is constant in $T$, we have
	\begin{equation*}
		\begin{aligned}
			\int_{T} \Wppf (\nabla I_{h}u - \nabla u) \nabla v_{h} \,\text{d}x = \Wppf \  \nabla v_{h}\vert_{T} \cdot \int_{T} (\nabla I_{h}u - \nabla u)\,\text{d}x= 0.
		\end{aligned}
	\end{equation*}
	
\end{proof}

\subsubsection{Coarsening error of external forces}
\label{sec: external_forces_qnll_cg}
We now address the consistency error arising from approximating the external potential $\langle f, v_{h} \rangle_{\Z}$ using the trapezoidal rule, denoted as $\langle f, v_{h} \rangle_{h}$. A key challenge in this analysis is to avoid relying on the Poincaré inequality $\Vert v_{h} \Vert_{L^{2}} \lesssim N\Vert \nabla v_{h} \Vert_{L^{2}}$. Instead, we utilize weighted Poincaré inequalities, which are more suitable for unbounded domains. This approach leads to the following result. The lemma presented here is adapted from~\cite[Theorem 6.13]{2013_ML_CO_AC_Coupling_ACTANUM}, but we state the theorem directly and omit the proof for brevity.

\begin{lemma}
	Let $L>1,\omega(x) = x\log(x)$ and suppose that $h(x)\le \kappa x$ for almost every $x\in [-N,-\bar{K}\cup[\bar{K},N]$. And we note $[-\infty,-\bar{K}]\cup[\bar{K},+\infty]$ by $\tOmeC$Then there exists a constant $C_{\kappa}$ such that
	\begin{equation*}
		\Vert \eta_{\text{ext}}\Vert_{(\YhNL)^{*}} = \Vert \langle f,\cdot\rangle_{N} - \langle f,\cdot\rangle_{h}\Vert_{(\YhNL)^{*}} \lesssim \Vert h^{2} \nabla f\Vert_{L^{2}(\tOmeC)} +\frac{C_{\kappa}}{\log L} \Vert h^{2}\omega \nabla^{2}f\Vert_{L^{2}(\tOmeC)}.
	\end{equation*}
\end{lemma}

\subsection{A priori existence and error estimate}
\label{sec: priori_anal_qnll_cg}

Next, we consider the a priori error estimation between the coarse-grained QNLL model and the atomistic model. Based on the inverse function theorem and Theorem~\ref{Internal forces of continuum region}, we can obtain the following result:

\begin{theorem}
	Let $\yai \in \Ya$ be a strongly stable atomistic solution satisfying \eqref{All-Atomistic strong local minimizer} and $\DH$. Consider the QNLL problem \eqref{Yh solution} with quasi-optimal choice of $N, \Th$. Suppose, moreover, that $\El$ is stable in the reference state Theorem \ref{Stability}. Then, there exists $K_0$ such that, for all $K \ge K_0$, \eqref{Yh solution} has a locally unique, strongly stable solution $y^{\text{NL-L}}_{h}$ which satisfies
	\begin{equation}
		\begin{aligned}
			\Vert \nabla\yai - \nabla y^{\text{NL-L}}_{h} \Vert_{L^2} \lesssim 8&M^{(3,0)}(\Vert \nabla^{2} u\Vert_{L^{2}(\bOmeI)} +\Vert \nabla^{3}u \Vert_{L^{2}(\bOmeC)}+\Vert \nabla^{2}u \Vert^{2}_{L^{4}(\bOmeC)}\\
			&+ \Vert \nabla u \Vert^{2}_{L^{4}(\OmeL)}+\Vert h \nabla^{2} u\Vert_{L^{2}(\OmeC)}+N^{\frac{1}{2}-\alpha})/\big(\min(c_{0},\ganllF)\big)^2.
		\end{aligned}
	\end{equation}
	
\end{theorem}
\begin{proof}
	
	The proof process here is similar to Theorem \ref{Priori of NCG}, with the difference being the inclusion of coarse-grained error in $\eta$ during the application of the inverse function theorem.

\end{proof}

\subsection{Discussion of the (quasi-)optimal choice of the length of the nonlinear and linear region}
\label{Balance of QNLL CG model}


In this subsection, we will discuss how to achieve the quasi-optimal choice of the lengths for finite element mesh $h$, nonlinear continuum region, linear continuum region, and computational region to obtain quasi-optimal convergence order for the QNLL model. We observe that due to this error balance, we only need a very short nonlinear region, this is the key motivation that we introduce this coupling methods to gain the same accuracy but a much more efficient method.

\subsubsection{Optimizing the finite element grid}
\label{sec: choice_of_fem_cg}
The finite mesh size $h$ is the first approximation parameter that we will optimize. In this section, we use a classical technique to optimize the mesh grading.

For each $x\in [-N,N], \ x\in \text{int} T$, let $h(x):=h_{T}$. For $x<-N$ or $x>N$, let $h(x):=1$. The coarse-grained error occurring in the coarsening analysis that depends on $\ThNL$ are the interpolation error term $\Vert h \nabla^{2}u\Vert_{L^{2}(\OmeC)}$. Suppose that $u\in\UhNL$ satisfies $\DH$ and $L >r_{0}$. Then
\begin{equation*}
	\Vert h \nabla^{2}u\Vert_{L^{2}(\OmeC)} \lesssim \Vert h x^{-\alpha-1}\Vert_{L^{2}(\OmeC)}.
\end{equation*}
We wish to choose $h$ to minimize this quantity, subject to fixing the number of degrees of freedom $\NhNL$, which is given by
\begin{equation*}
	\NhNL=\sum_{j=1}^{\NhNL}1=\sum_{j=1}^{\NhNL} h_{j}\frac{1}{h_{j}}=\int_{-N}^{N}\frac{1}{h} \,\d x.
\end{equation*}
We ignore the discreteness of the mesh size function and solve
\begin{equation*}
	\min \Vert h x^{-\alpha-1}\Vert_{L^{2}(\OmeC)} \quad \text{subject to }\int_{-N}^{N} \frac{1}{h}\,\d x = \text{const}.
\end{equation*}
The solution to this variational problem satisfies
\begin{equation*}
	h(x)=\lambda\vert x\vert^{\frac{2}{3}(\alpha+1)} \ \text{for }x\in \OmeC.
\end{equation*}
for some constant $\lambda>0$. This gives us an optimal scaling of the mesh size function.

We now impose the condition $h(L)\approx 1$, which yields
\begin{equation}\label{Mesh size fucntion of nonlinear h}
	h^{\text{NL}}(x)\approx(\frac{\vert x\vert}{L})^{\frac{2}{3}(\alpha+1)}=:\tilde{h}(x) \ \text{for }x\in\OmeC.
\end{equation}

If $\alpha':=\frac{2}{3}(\alpha+1)$, then $\alpha'>1$, and hence this choice of $h$ yields(for simplify, we only calculate the domain$[\bar{K},N]$)
\begin{equation}\label{int of nonlinear h}
	\int_{\bar{K}}^{N} \frac{1}{h} \,\d x \approx \frac{\bar{K}^{\alpha'}(N^{1-\alpha'}-\bar{K}^{1-\alpha'})}{1-\alpha'}\approx \frac{\bar{K}}{\alpha'-1}.
\end{equation}

Thus, the choice \eqref{Mesh size fucntion of nonlinear h} gives a comparable number of degrees of freedom in the linear region to that in the atomistic, interface and continuum regions. The resulting interpolation error bound can be estimated by
\begin{equation}\label{Decay of the nonlinear best approximation term}
	\Vert h x^{-\alpha-1}\Vert_{L^{2}(\OmeC)} \approx \frac{\bar{K}^{\frac{1}{2}-(\alpha+1)}}{(\alpha'-1)^{\frac{1}{2}}} \approx \bar{K}^{-\frac{1}{2}-\alpha}.
\end{equation}

\subsubsection{The quasi-optimal choice of $N$}
\label{sec: choice_of_N_cg}

Before introducing specific finite element mesh generation algorithm, we need to discuss how to determine the length of our computational domain $N$. We follow two principles:
\begin{enumerate}
	\item We should ensure that the truncation error term $N^{\frac{1}{2}-\alpha}$ do not dominate among the various types of errors after balancing the length of the computational domain;
	
	\item We choose the length of the computational domain as small as possible for computational simplicity.
\end{enumerate}

According to the first principle mentioned above, we understand that the truncation error $N^{\frac{1}{2}-\alpha}$ must be balanced against one of the terms of modelling error or coarsening error (or higher-order terms). According to the second principle, to select the computational domain length as small as possible, we must balance it against the lowest-order term of modelling error or coarsening error (balancing against higher-order terms would need a longer computational domain length).

For $\frac{1}{2}< \alpha< 1$, the lowest-order term is the linearization error $\Vert \nabla u \Vert^{2}_{L^{4}(\OmeL)}\lesssim L^{\frac{1}{2}-2\alpha}$, we should choose $N$ such that
\begin{equation*}
	L^{\frac{1}{2}-2\alpha}\approx N^{\frac{1}{2}-\alpha}, \quad \text{that is}, \ N\approx L^{\frac{2\alpha-1/2}{\alpha-1/2}}.
\end{equation*}

For $ \alpha\ge 1$, the lowest-order term is the coarse-grained error $\Vert h\nabla^2 u \Vert_{L^{2}(\OmeC)}\lesssim \bar{K}^{-\frac{1}{2}-\alpha}$, we choose $N$ such that
\begin{equation*}
	\bar{K}^{-\frac{1}{2}-\alpha}\approx N^{\frac{1}{2}-\alpha}, \quad \text{that is}, \ N\approx \bar{K}^{\frac{\alpha+1/2}{\alpha-1/2}}.
\end{equation*}

We now turn this formal motivation into an explicit construction of the finite element mesh. 
	
	

\begin{algorithm}[H]
	\caption{Finite element mesh construction algorithm}
	\label{alg:FEM}
	\begin{enumerate}
		\item[Step 1]: Set $N:= \lceil \bar{K}^{\frac{\alpha+1/2}{\alpha-1/2}}\rceil$($\alpha>1$); $N:= \lceil \bar{K}^{\frac{2\alpha-1/2}{\alpha-1/2}}\rceil$($\frac{1}{2}< \alpha< 1$). 
		\item[Step 2]: Set $\NhNL :=\{0, 1, \dots, \bar{K}\}$.  		
		\item[Step 3]: While $n:=\max (\NhNL) <L$:
		\begin{enumerate}
			\item[Step 3.1]: Set $\NhNL:=\NhNL \cup \{\min(L,n+\lfloor \tilde{h}(n)\rfloor ) \}$.
		\end{enumerate}
		\item[Step 4]: While $n:=\max (\NhNL) <N$:
		\begin{enumerate}
			\item[Step 4.1]: Set $\NhNL:=\NhNL \cup \{\min(N,n+\lfloor \tilde{h}(n)\rfloor ) \}$.
		\end{enumerate}
		\item[Step 5]: Set $\NhNL = (-\NhNL) \cup \NhNL$.
	\end{enumerate}
\end{algorithm}






Meshes construct via this algorithm qualitatively the same properties as predicted by the formal computations \eqref{Mesh size fucntion of nonlinear h} and \eqref{int of nonlinear h}.
\begin{theorem}\label{Decay result of nonlinear best approximation term}
	Let $u\in \Ua$ satisfy $\DH$ and let $N$ and $\ThNL$ be constructed by Algorithm \ref{alg:FEM}. Then, for $L$ sufficiently large, $\NhNL\le C_{1} L$,
	\begin{equation*}
		\begin{aligned}
			\Vert h\nabla^{2} u\Vert_{L^{2}(\OmeC)} &\le C_{2} \NhNL^{-\frac{1}{2}-\alpha},  \\
			\Vert h\nabla^{2} u\Vert^{2}_{L^{4}(\OmeNL)} &\le C_{3} \NhNL^{-\frac{1}{2}-2\alpha},
		\end{aligned}
	\end{equation*}
	where $C_{1}$ depends on $\alpha$ and $C_{2}, C_{3}$ depends on $\alpha$ and on $\CDH$.
\end{theorem}

We now turn the external consistency error estimate into an estimate in terms of $\NhNL$ as well. Let $f$ satisfy $\DH$ and suppose that $\ThNL$ and $N$ are constructed using Algorithm \ref{alg:FEM}. Since $h(x)\le \frac{x}{2}, \ \omega(x)=x\log x$, a straightforward computation yields
\begin{align*}
	\Vert \eta_{\text{ext}} \Vert_{(\YhNL)^{*}} &\lesssim \Vert h^{2} \nabla f\Vert_{L^{2}(\tOmeC)} +\frac{C_{\kappa}}{\log \bar{K}} \Vert h^{2}\omega \nabla^{2}f\Vert_{L^{2}(\tOmeC)}\\
	&\lesssim \bar{K}^{-\alpha-\frac{3}{2}}+\frac{\bar{K}^{-\alpha-\frac{3}{2}}\log^{2}N}{\log \bar{K}}.
\end{align*}

We insert $N \lesssim \bar{K}^{\frac{2}{3}(\alpha+1)}$ to obtain the following result. In particular, we can conclude that the external consistency error is dominated by the interpolation error.

\begin{theorem}\label{Decay result of nonlinear external force}
	Let $f$ satisfy $\DH$ and let $\ThNL$, $N$ be constructed by Algorithm \ref{alg:FEM}. Then
	\begin{equation*}
		\Vert \eta_{\text{ext}} \Vert_{(\YhNL)^{*}} \le C_{\alpha} \bar{K}^{-\alpha-\frac{3}{2}} \log \bar{K}.
	\end{equation*}
	where $C_{\alpha}$ depends on $\alpha$ and on $\CDH$.
\end{theorem}

\subsubsection{The quasi-optimal choice of $L$}
\label{sec: choice_of_L_cg}
We consider the nonlinear-linear elasticity coupling method quasi-optimal choice of approximation parameters $ \bar{K}, L, \NhNL$, for given $K$. For the reason we choose $\rcut=2$, we fix $\bar{K}=K+2$. So we wanna balance $\bar{K}, L$ and $\NhNL$. And the key idea is how to choose $L$ to balance the lowest  order term between $\bar{K},L$.

Firstly, we use $\DH$ assumption, and we could obtain the decay result of coupling error
\begin{equation*}
	\begin{split}
		&M^{(2,1)}\Vert \nabla^{2} u\Vert_{L^{2}(\bOmeI)} +M^{(2,2)}\Vert \nabla^{3}u \Vert_{L^{2}(\bOmeC)}+M^{(3,2)}\Vert \nabla^{2}u \Vert^{2}_{L^{4}(\bOmeC)} + M^{(3,0)}\Vert \nabla u \Vert^{2}_{L^{4}(\OmeL)}\\
		&\lesssim \CDH M^{(2,1)}K^{-\alpha -1}+\CDH M^{(2,2)}\bar{K}^{-\alpha-\frac{3}{2}} +\CDH^{2}M^{(3,2)}\bar{K}^{-2\alpha -\frac{3}{2}}+\CDH^{2}M^{(3,0)}L^{-2\alpha+\frac{1}{2}} .
	\end{split}
\end{equation*}

For coarsening error, we have assumption $\NhNL \lesssim \bar{K}$. From Theorem \ref{Decay result of nonlinear best approximation term} and Theorem \ref{Decay result of nonlinear external force}, we know the interpolation error $\Vert h\nabla^{2}u\Vert_{L^{2}(\OmeC)}$ is the dominant contribution
\begin{equation*}
	\Vert h\nabla^{2} u\Vert_{L^{2}(\OmeC)} \le C_{2}^{\text{NL}} \NhNL^{-\frac{1}{2}-\alpha}.
\end{equation*}

The lowest-order term of $L$ is $\Vert \nabla u \Vert^{2}_{L^{4}(\OmeL)} \lesssim L^{-2\alpha+\frac{1}{2}}$. We balance this term with $	\Vert h\nabla^{2} u\Vert_{L^{2}(\OmeC)} \lesssim \bar{K}^{-\frac{1}{2}-\alpha}$, and get (by noticing the fact that $\bar{K}\le L$)
\begin{align}
	L &\lesssim \bar{K}^{\frac{1}{2}+\frac{3}{8\alpha-2}} \qquad (\frac{1}{2}<\alpha<1) \label{Balance of L CG 1},\\
	L &\approx \bar{K} \qquad \qquad \quad (\alpha \ge 1)\label{Balance of L CG 2}.
\end{align}

\subsection{Numerical validation}
\label{sec: experiments_qnll_cg}

In this section, we present numerical experiments to illustrate our analysis. Same as in Section \ref{sec: experiments_qnll_ncg}, the problem is a typical one-dimensional test case, with the site energy modeled using the embedded atom method (EAM), a widely used atomistic model for solids. We fix the exact solution as defined in Section \ref{sec: experiments_qnll_ncg} and compute the external forces, which are equal to the internal forces under the deformation. The decay exponent ensures that the solution and forces satisfy the decay hypothesis $\DH$.

In this section, We will demonstrate the method of controlling the length of non-linear continuum region in the QNLL model to achieve quasi-optimal convergence order, as introduced in Section \ref{Balance of QNLL CG model}. We will conduct numerical experiments with the atomistic model length of 100,000 atoms.We set energy functional and external force to \eqref{EAM of numerical experiments} and \eqref{External force of numerical experiments}, and using Algorithm \ref{alg:FEM} to construct finite element mesh. The experiments will be carried out for $\alpha$ values of $0.8, 1.0$ and $1.2$.

Firstly, let's consider the experiment with alpha set to 1.2: In this case, according to \eqref{Balance of L CG 2}, by setting the length of the nonlinear continuum region to a few atoms ($\bar{K} \approx L$), the convergence order of the QNLL method matches that of the QNL method. In the Figure below, the $x$-axis represents the degrees of freedom (dof) in mesh, while the $y$-axis shows the absolute error $\Vert \nabla \yai - \nabla y^{\text{ac}} \Vert_{L^{2}}~\text{(ac} = \text{QNL, QNLL)}$ between the reference atomistic solution $\yai$ and the a/c solutions $y^{\text{ac}}$. It can be observed that the two convergence order lines in the graph nearly overlap, indicating that the difference between the two AC solutions $\Vert \nabla y^{\text{QNL}} - \nabla y^{\text{QNLL}}\Vert$ is between $10^{-7} $and $10^{-8}$.

\begin{figure}[h]
	\centering 
	\includegraphics[width=0.6\textwidth]{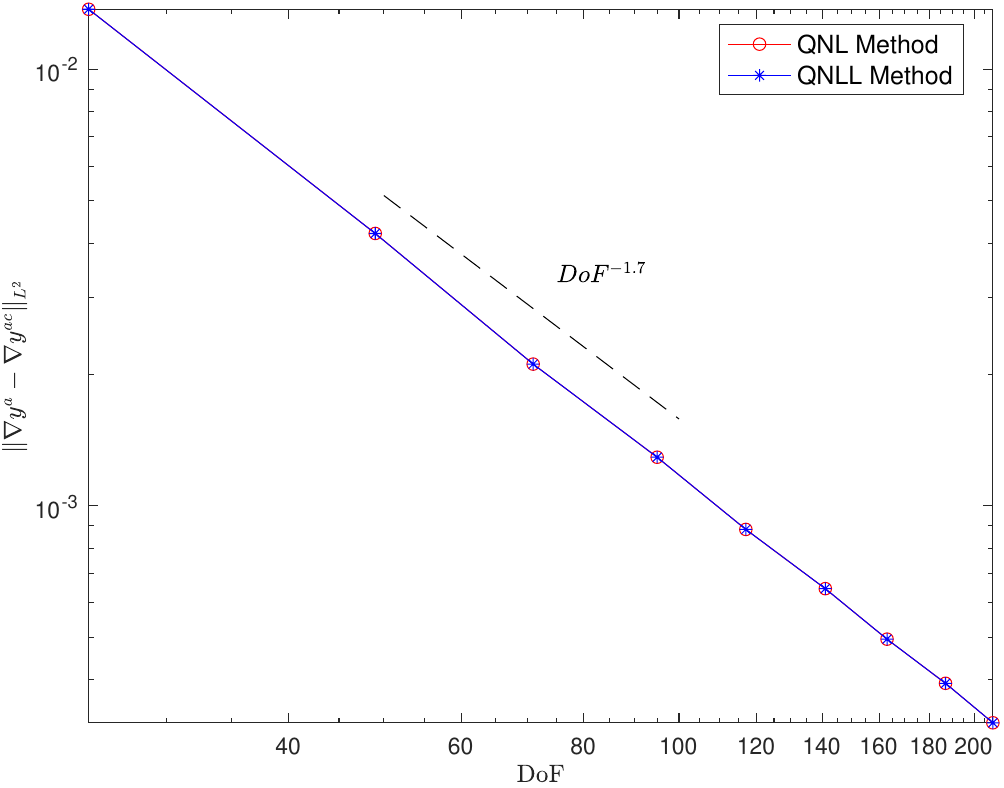}
	\caption{The convergence order of QNL and QNLL method ($\alpha = 1.2$)} 
	\label{fig: convergence_QNL_QNLL_alpha12_CG}
\end{figure}

Next, to demonstrate the computational efficiency of the QNLL method, we test the variation in computation time by progressively increasing the degrees of freedom of the nonlinear continuum region $ \Nnl$ of the QNLL method, while keeping the finite element mesh fixed, meaning the continuum region remains unchanged. The results are as shown in the table below: the first column lists the method names, with parentheses indicating the proportion of the degrees of the freedom of nonlinear continuum region$ \Nnl $ to that of the total continuum region$ \Nc$ and the second column records the ratio of the computing time of the QNLL method to the computing time of the QNL method on a device the same as in Section \ref{sec: experiments_qnll_ncg}.


\begin{table}
	\centering
	\renewcommand{\arraystretch}{1.5} 
	\begin{tabular}{|c|c|} 
		\hline 
		Method ($\Nnl/\Nc$) & The ratio of the computing time\\ 
		\hline 
		QNLL ($19.84\%$) & $69.38\%$ \\ 
		QNLL ($48.19\%$) & $77.60\%$ \\ 
		QNLL ($81.26\%$) & $87.95\%$ \\ 
		QNL ($100\%$) & $100\%$ \\ 
		\hline 
	\end{tabular}
	\caption{The computing time (with coarse graining) of QNL and QNLL method ($\alpha = 1.2$), with degree of freedom set to 1377 for all methods.}
	\label{tab:computing time alpha12 CG}
\end{table}


According to the Table \ref{tab:computing time alpha12 CG}, we can see that as the proportion of the nonlinear continuum region length to the total continuum region length increases, computing time clearly increases. However, in practical applications, the proportion of nonlinear elements will be lower (below 5$\%$) according to the balancing method described in Section \ref{Balance of QNLL CG model}. The ratio of the difference between the absolute errors of the QNLL method and the QNL method to the absolute errors of the QNL method: $( \Vert \nabla \yai - \nabla y^{\text{QNLL}} \Vert_{L^{2}} - \Vert \nabla \yai - \nabla y^{\text{QNL}} \Vert_{L^{2}}) / \Vert \nabla \yai - \nabla y^{\text{QNL}} \Vert_{L^{2}}$ is in a narrow range. Here, the ratio, as defined above, is within the range of $10^{-5}$ to $10^{-6}$. This indicates that the QNLL method maintains high accuracy while still offering computational efficiency advantages.

When $\alpha = 1$, the results are similar to when $\alpha = 1.2$. The following figure compares the convergence order of the QNLL method and the QNL method. The information represented on the axes is the same as in Figure \ref{fig: convergence_QNL_QNLL_alpha12_CG}. We observe a similar outcome to Figure \ref{fig: convergence_QNL_QNLL_alpha12_CG}, where the convergence lines of the QNLL method closely overlap with those of the QNL method.

\begin{figure}[h]
	\centering 
	\includegraphics[width=0.6\textwidth]{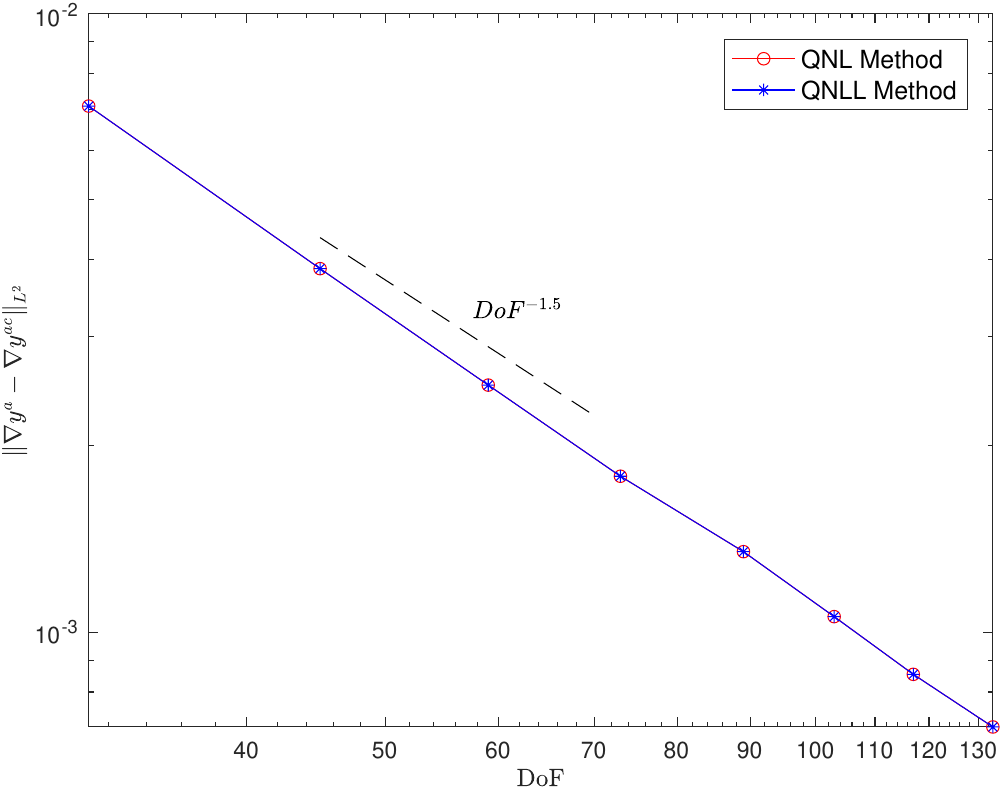}
	\caption{The convergence order of QNL and QNLL method ($\alpha = 1.0$)} 
	\label{fig: convergence_QNL_QNLL_alpha10_CG}
\end{figure}

Similar to the case when $\alpha = 0.8$, we test the gradual increase in length of the nonlinear continuum region under the condition of a fixed finite element mesh. The remaining configurations and the information represented in each column are the same as in Table  \ref{tab:computing time alpha12 CG}. We observed results similar to Table \ref{tab:computing time alpha12 CG}, where as the proportion of the nonlinear continuum region length to the continuum region length increases gradually, the computing time also increases gradually, but there is no significant reduction in error.

\begin{table}
	\centering
	\renewcommand{\arraystretch}{1.5} 
	\begin{tabular}{|c|c|} 
		\hline 
		Method ($\Nnl/\Nc$) & The ratio of the computing time\\ 
		\hline 
		QNLL ($27.89\%$) & $78.00\%$ \\ 
		QNLL ($49.67\%$) & $84.46\%$ \\ 
		QNLL ($76.91\%$) & $92.18\%$ \\ 
		QNL ($100\%$) & $100\%$ \\ 
		\hline 
	\end{tabular}
	\caption{The computing time (with coarse graining) of QNL and QNLL method ($\alpha = 1.0$), with degree of freedom set to 1943 for all methods.}
	\label{tab:computing time alpha10 CG}
\end{table}

Furthermore, we will now consider the case where $\alpha=0.8$. In this setting, according to \eqref{Balance of L CG 1} and \eqref{Balance of L CG 2}, there are two finite element mesh generation schemes for the QNLL method:
\begin{enumerate}
	\item In the first scheme, we focus on the accuracy of the QNLL method. According to \eqref{Balance of L CG 1}, we precisely balance the atomistic region, nonlinear continuum region, linear continuum region, and the total length of the computational domain to achieve convergence order identical to those of the QNL method.
	
	\item In the second scheme, we prioritize the computational efficiency of the QNLL method. Therefore, after balancing the lengths of the atomistic region and the total length of the computational domain, we minimize the length of the nonlinear continuum region as much as possible, even down to just a few atoms.
\end{enumerate}

In Figure~\ref{fig: convergence_QNL_QNLL_alpha08_CG}, we represent the first finite element mesh generation scheme with red dashed squares for the QNLL method, and the second generation scheme with blue dashed squares. To demonstrate the accuracy of the QNLL method, the QNL method also adopts the first finite element mesh generation scheme, depicted in the figure with red dashed star symbols. The information represent on the axes is the same as in Figure \ref{fig: convergence_QNL_QNLL_alpha12_CG}. We observe that, after balancing the lengths of the atomistic region, nonlinear continuum region, linear continuum region, and the total length of the computational domain, the absolute errors and convergence order obtained by the QNLL method are consistent with those of the QNL method. However, after reducing the length of the nonlinear continuum region in pursuit of computational efficiency, there is a noticeable increase in absolute errors and a decrease in convergence speed.

\begin{figure}
	\centering 
	\includegraphics[width=0.6\textwidth]{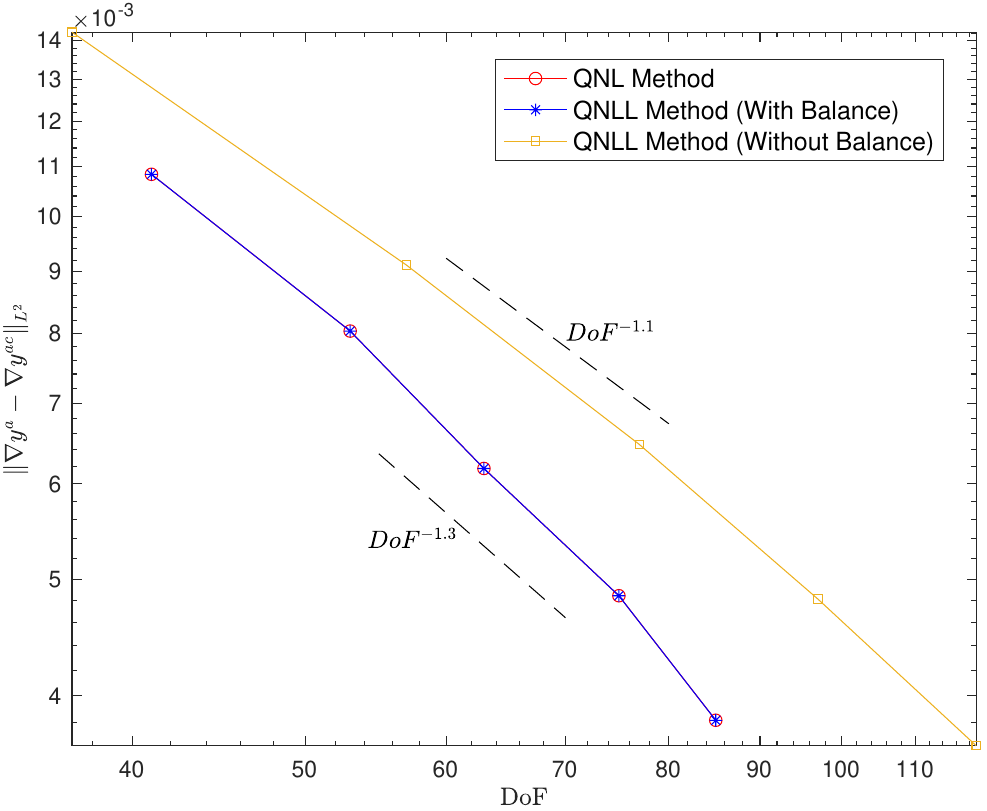}
	\caption{The convergence order of QNL and QNLL method ($\alpha = 0.8$)} 
	\label{fig: convergence_QNL_QNLL_alpha08_CG}
\end{figure}

The following table displays the changes in computation time as the length of the nonlinear continuum region gradually increase, with $\alpha$ set to 0.8. The remaining configurations and the information represented in each column are the same as in Table \ref{tab:computing time alpha12 CG}. Similar to Table \ref{tab:computing time alpha12 CG}, we observe that as the proportion of the nonlinear continuum region length to the continuum region length increases, computing time gradually increases, but there is no significant reduction in error.

\begin{table}[h]
	\centering
	\renewcommand{\arraystretch}{1.5} 
	\begin{tabular}{|c|c|} 
		\hline 
		Method ($\Nnl/\Nc$) & The ratio of the computing time\\ 
		\hline 
		QNLL ($28.25\%$) & $68.50\%$ \\ 
		QNLL ($49.13\%$) & $73.26\%$ \\ 
		QNLL ($76.97\%$) & $80.08\%$ \\ 
		QNL ($100\%$) & $100\%$ \\ 
		\hline 
	\end{tabular}
	\caption{The computing time (with coarse graining) of QNL and QNLL method ($\alpha = 0.8$), with Degree of Freedom (DoF) set to 2981 for all methods.}
	\label{tab:computing time alpha08 CG}
\end{table}

	
	\section{Conclusion and Future Work}
\label{sec: conclusion}

In this work, we incorporate the linear elasticity model to enhance the computational efficiency of the QNL method. Specifically, we introduce a coupling energy formulation that integrates the nonlinear QNL method with the linearized Cauchy-Born model, referred to as the QNLL method. Through a rigorous {\it a priori} error analysis, we demonstrate that the convergence of the QNLL method depends on the balance between the lengths of the computational, atomistic, nonlinear continuum, and linear continuum regions, as well as the finite element coarse-graining. Our analysis ensures that the QNLL method achieves the same convergence order as the original QNL method while significantly reducing computational cost. Numerical experiments validate these theoretical findings, highlighting the efficiency and practical advantages of the QNLL method.

However, several problems remain open which deserve further investigation.


\begin{itemize}
    \item {\it a Posteriori Error Control:} A natural extension of our approach is the development of \textit{a posteriori} error control methods, which have a rich body of literature~\cite{1996_RV_A_Post_Adapt,2014_CO_HW_A_Post_ACC_IMANUM, 2023_YW_HW_Efficient_Adaptivity_AC_JSC,2021_YW_HC_ML_CO_HW_LZ_A_Post_QMMM_SISC}. The main challenge in this context lies in adaptively selecting the lengths of the regions and constructing the corresponding error estimators or indicators for these lengths. Several strategies, such as balancing error estimates with mesh refinement techniques, could be explored to improve computational efficiency while maintaining accuracy in the error bounds. This extension is crucial for future work, where the focus will be on optimizing these error indicators and integrating them effectively into multiscale models.
    
	\item {\it Extension to Higher Dimensions:} Extending our framework to higher dimensions requires a shift from QNL to GRAC methods~\cite{2012_CO_LZ_GRAC_Construction_SIAMNUM,2014_CO_LZ_GRAC_Coeff_Optim_CMAME}. The introduction of linearized Cauchy-Born in this context raises significant concerns regarding the linearization errors, especially when extending to higher dimensions. Specifically, challenges arise in quantifying how these linearization errors propagate through the coupled system, and ensuring the stability of the method in higher dimensions is a direct issue. Further research is needed to assess the extent of these linearization errors in higher-dimensional settings and to determine whether additional modifications or error control techniques are required to maintain robustness and computational efficiency in the model.

	\item {\it Generalization to Other a/c Coupling Methods:} Expanding our approach to incorporate other a/c coupling techniques, such as blending methods, is an important next step~\cite{2008_SB_MP_PB_MG_AC_Blending_MMS,2011_BK_ML_BQCE_1D_SIMNUM,2016_XL_CO_AS_BK_BQC_Anal_2D_NUMMATH, wang2023adaptive}. The introduction of linearized Cauchy-Born in these methods requires further exploration, particularly with regard to how nonlinear to linear transitions are handled. A critical question is whether blending is necessary to manage this transition smoothly. Additionally, the force-based method formulation in the context of linearized Cauchy-Born is worth investigating~\cite{2008_MD_ML_Ana_Force_Based_QC_M2NA,2019_HW_SL_FY_A_Post_QCF_1D_NMTMA}, as this approach may offer significant computational speedup by reducing the problem to linear equations in the continuum region.
	
	\item {\it Theoretical Estimates for Computational Efficiency Improvements:} While we provide the priori anlysis of our method, further work is needed to quantify the theoretical estimate in computational efficiency. Specifically, optimization techniques can be applied to estimate how the proposed multiscale framework could lead to cheaper computational cost. Future research could focus on developing a theoretical framework for computational cost reduction based on the error estimates and meshing strategies, offering a more complete understanding of how to improve computational efficiency alongside accuracy.

    \item {\it Possible Connections with Sequential Multiscale Methods:} Although our current focus is on concurrent multiscale methods, future work should consider potential connections with sequential methods, particularly force-based methods like the Flexible Boundary Condition (FBC) Method~\cite{2002_CW_SR_FBC_Dislocation_PHYS,2021_MH_Anal_FBC_AC_MMS,2008_RT_LGF_Long-range_PHYS}. The FBC method, while widely used for simulating defects in materials, still faces challenges such as slow convergence speed. Our approach may offer useful insights for the convergence analysis of FBC methods, particularly regarding error propagation and the balance between the lengths of regions. A key area of future investigation is the adaptation of existing sequential methods by improving convergence rates, such as extending the relaxation region and applying finite element coarsening techniques. Furthermore, exploring hybrid methods that integrate both sequential and concurrent techniques could lead to new strategies that harness the strengths of both approaches, improving both accuracy and computational efficiency.

\end{itemize}

Both the theoretical and practical aspects discussed above will be explored in future work.

	\appendix
	\renewcommand\thesection{\appendixname~\Alph{section}}
	
	\section{Proof}
\label{sec: appendix}
\renewcommand{\theequation}{A.\arabic{equation}}

In this section, we first state the well-known inverse function theorem, which plays a crucial role in the {\it a priori} error estimates presented in this work. The result is a simplified and specialized version of~\cite[Lemma 2.2]{2011_CO_1D_QNL_MATHCOMP}, though similar formulations can be derived from standard proofs of the inverse function theorem.

\begin{lemma}\label{Inverse function theorem}
	\textbf{Inverse function theorem: }Let $\Yn$ be a subspace of $\Ya$, equipped with $\Vert \nabla \cdot \Vert_{L^{2}}$, and let $\Ghc \in C^{1}(\Yn,\Yn^{*})$ with Lipschitz-continuous derivative $\delta \Ghc$:
	\begin{equation*}
		\Vert\delta\Ghc(y) - \delta\Ghc(v)\Vert_\mathcal{L} \le M \Vert \nabla y - \nabla v \Vert_{L^{2}}, \quad \text{for all} \  v \in \Un,
	\end{equation*}
	where $\Vert \cdot \Vert_{\mathcal{L}}$ denotes the $\mathcal{L}(\Yn,\Yn^{*})$-operator norm.
	
	Let $\bar{y}\in \Yn$ satisfy
	\begin{align}
		\Vert \Ghc(\bar{y})\Vert_{\Yn^{*}} &\le \eta,\\
		\langle \delta \Ghc (\bar{y})v,v\rangle &\ge \gamma \Vert \nabla v \Vert^{2}_{L^{2}},  \quad \text{for all} \ v \in \UhNL,
	\end{align}
	such that $L,\eta,\gamma$ satisfy the relation
	\begin{equation}
		\frac{2M\eta}{\gamma^{2}}<1.
	\end{equation}
	
	Then there exists a (locally unique) $\ynll\in \Yn$ such that $\Ghc(\ynll)=0$,
	\begin{align}
		\Vert \nabla \ynll-\nabla \bar{y}\Vert_{L^{2}} &\le 2\frac{\eta}{\gamma}, \quad \text{and}\\
		\langle \delta \Ghc (\ynll)v,v\rangle &\ge (1-\frac{2M\eta}{\gamma^{2}}) \Vert \nabla v \Vert^{2}_{L^{2}},  \quad \text{for all} \ v \in \Un,
	\end{align}
\end{lemma}

\subsection{Proof of Proposition \ref{Pointwise coupling stress tensor}}
\label{Appendix section 1}

\begin{proof}
	The proof of the coupling error estimate for the QNL model is derived from [\cite{2013_ML_CO_AC_Coupling_ACTANUM}, Lemma 6.12]. Since $\text{supp}(\Ki) = \conv\{\xi,\xi+\rho\}$, $\Rrfl=0$ for $x\in\OmeA \backslash \bOmeA$. As for $x\in \OmeI \cup \bOmeA$, we obtain
	\begin{align*}
		\Rrfl &=\Srfl-\Sa\\
		&=\sum_{\xi \in \Ic} \sum_{\rho \in \Rc}\rho\Ki\Phii_{\xi,\rho}(y)-\sum_{\xi \in\Ic\cup\Cc}\sum_{\rho \in \Rc} \rho \Ki \Phia_{\xi,\rho}(y).
	\end{align*}
	We define uniform deformation $\yF(\xi)=F\xi$. We notice that $R^{\text{rfl}}(\yF;x)=0$, hence we obtain
	\begin{equation}\label{Two parts of Rrfl}
		\begin{aligned}
			\Rrfl &= \Rrfl-R^{\text{rfl}}(\yF;x)\\
			&= \sum_{\xi \in \Ic} \sum_{\rho \in \Rc}\rho\Ki\big(\Phii_{\xi,\rho}(y)-\Phii_{\xi,\rho}(\yF)\big)\\
			&\ +\sum_{\xi \in\Ic\cup\Cc}\sum_{\rho \in \Rc} \rho \Ki \big(\Phia_{\xi,\rho}(\yF)-\Phia_{\xi,\rho}(y)\big).
		\end{aligned}
	\end{equation}
	
	For the first term of \eqref{Two parts of Rrfl}, we choose $F=\nabla y$ and use Taylor's expansion. We have
	\begin{equation}\label{Interface stess}
		\vert \Phii_{\xi,\rho}(y)-\Phii_{\xi,\rho}(\yF) \vert \le \sum_{\zeta \in \Rc} c(\rho,\zeta)\vert D_{\zeta}y(\xi)-\nabla_{\zeta}y(x)\vert.
	\end{equation}
	After using Taylor's expansion, for $x\in \conv\{\xi,\xi+\rho\}$, we obtain
	\begin{align*}
		\vert D_{\zeta}y(\xi)-\nabla_{\zeta}y(x)\vert&= \vert \zeta(\xi-x+\frac{\zeta}{2})\vert \cdot \Vert \nabla^{2}y\Vert_{L^{\infty}(\conv\{\xi,\xi+\rho\})}\\
		&\le\frac{\vert \zeta \vert^{2}}{2}\Vert \nabla^{2}u\Vert_{L^{\infty}(\conv\{\xi,\xi+\rho\})}.
	\end{align*}
	It is straightforward to calculate that
	\begin{align*}
		\sum_{\xi \in \Ic} \sum_{\rho \in \Rc} \vert\rho \vert\Ki \vert \Phii_{\xi,\rho}(y)-\Phii_{\xi,\rho}(\yF) \vert &\le \sum_{(\rho,\zeta)\in\Rc^{2}} \frac{1}{2}\vert \rho \zeta^{2}\vert c(\rho,\zeta) \Vert \nabla^{2}y\Vert_{L^{\infty}(v_{x})}\sum_{\xi \in \Ic}\Ki\\
		&\lesssim M^{(2,1)}\Vert \nabla^{2}y\Vert_{L^{\infty}(v_{x})}.
	\end{align*}
	
	Similarly, we could calculate that
	\begin{align*}
		\sum_{\xi \in\Ic\cup\Cc} \sum_{\rho \in \Rc} \vert\rho \vert\Ki \vert \Phia_{\xi,\rho}(y)-\Phia_{\xi,\rho}(\yF) \vert &\le \sum_{\rho \in \Rc} \frac{1}{2} \vert \rho \zeta^{2}\vert m(\rho,\zeta) \\
		&\le M^{(2,1)} \Vert \nabla^{2}y\Vert_{L^{\infty}(v_{x})}.
	\end{align*}
	
	For $x\in \OmeC \cap \Z+\frac{1}{2}$, we have
	\begin{equation*}
		\Rrfl = \partial_{F}W\big(\nabla y(x)\big)-\Sa,
	\end{equation*}
	which is the difference between the atomistic stress tensor and Cauchy-Born stress tensor. And we mention that the error estimate follows directly from[\cite{2013_ML_CO_AC_Coupling_ACTANUM}, Theorem6.2]
	\begin{align*}
		\Rrfl &= \partial_{F}W\big(\nabla y(x)\big)-\Sa\\
		&\lesssim M^{(2,2)}\Vert \nabla^{3}y\Vert_{L^{\infty}(v_{x})}+M^{(3,2)}\Vert \nabla^{3}y\Vert_{L^{\infty}(v_{x})}^{2}.
	\end{align*}
	This yields the stated results by noticing that $\nabla^{2}y =\nabla^{2}u,\ \nabla^{3}y=\nabla^{3}u$.
\end{proof}

\subsection{Proof of Proposition \ref{Coupling consistency error estimate}}\label{Appendix section 2}
\begin{proof}
	For \eqref{Interface region stress tensor}, the main point of this proof is to use the inverse estimates to obtain $L^{2}$-type from the $L^{\infty}$ bounds~\cite{2007_DB_FEM}
	\begin{equation}\label{L-infty to L-2 estimate}
		\Vert \nabla^{2}u\Vert_{L^{\infty}(v_{x})}\lesssim \Vert \nabla^{2}u\Vert_{L^{2}(v_{x})}.
	\end{equation}
	
	After a direct calculation, we have
	\begin{align*}
		\int_{\OmeI\cup\bOmeA}\Rrfl\nabla v\,\d x&\lesssim \int_{\OmeI\cup\bOmeA} M^{(2,1)}\Vert \nabla^{2}u\Vert_{L^{\infty}(v_{x})} \vert \nabla y \vert \,\d x\\
		&\le M^{(2,1)}\Vert \nabla^{2}u\Vert_{L^{\infty}(\bOmeI)}\Vert \nabla v \Vert_{L^{1}(\OmeI\cup\bOmeA)}\\
		&\lesssim  M^{(2,1)}\Vert \nabla^{2}u\Vert_{L^{\infty}(\bOmeI)}\Vert \nabla v \Vert_{L^{2}(\OmeI\cup\bOmeA)}.
	\end{align*}
	
	As for \eqref{Continuum region stress tensor}, combining Proposition \ref{Pointwise coupling stress tensor} and [\cite{2013_ML_CO_AC_Coupling_ACTANUM}, Corollary6.4], we yield the started results.
\end{proof}

\subsection{Proof of Lemma \ref{Pointwise blending lemma}}\label{Appendix section 3}
\begin{proof}
	Let $\psi(x):=\sqrt{1-\beta(x)}$ and assume, without loss of generality that $\rho>0$. Then,
	\begin{align*}
		\sqrt{1-\beta(\xi)} D_{\rho}v(\xi)&= \psi (\xi) \sum_{\eta = \xi}^{\xi+\rho-1}D_{1}V(\eta)\\
		&=\sum_{\eta = \xi}^{\xi+\rho-1}\psi(\eta)D_{1}V(\eta) + \sum_{\eta = \xi}^{\xi+\rho-1} \big(\psi(\xi)-\psi(\eta)\big)D_{1}V(\eta).\\
	\end{align*}
	If we define $\va$ by $D_{1}\va(\eta)=\psi(\eta)D_{1}v(\eta)$, then we obtain
	\begin{equation*}
		\sum_{\eta = \xi}^{\xi+\rho-1} \psi(\eta) D_{1}v(\eta) =D_{\rho} \va (\xi),
	\end{equation*}
	and after using Holder's inequality we know
	\begin{align*}
		\vert\sqrt{1-\beta} D_{\rho}v(\xi)-D_{\rho}\va(\xi)\vert&=\sum_{\eta = \xi}^{\xi+\rho-1}\big(\psi(\xi)-\psi(\eta)\big)D_{1}v(\eta)\\
		&\le (\sum_{\eta = \xi}^{\xi+\rho-1} \Vert \nabla \psi \Vert^{2}_{L^{\infty}} \vert \rho \vert^{2})^{\frac{1}{2}} (\sum_{\eta = \xi}^{\xi+\rho-1}\big(D_{1}v(\eta)\big)^{2})^{\frac{1}{2}}\\
		&\le \vert \rho \vert^{\frac{3}{2}} \Vert \nabla \psi \Vert_{L^{\infty}} \Vert \nabla v \Vert_{L^{2}(\xi,\xi+\rho)}.
	\end{align*}
	This establishes \eqref{Pointwise va blending estimate}. The proof of \eqref{Pointwise vc blending estimate1} is analogous, with $\vc$ defined by $D_{1}\vc(\xi) = \sqrt{\beta(\xi)}D_{1}v(\xi)$. With these definitions, \eqref{va and vc} is an immediate consequence.
\end{proof}
	
	
	

	\bibliographystyle{plain}
	\bibliography{MS_Coupling_202502.bib}
	
\end{document}